\documentclass[11pt,a4paper]{article} 

\usepackage[latin1]{inputenc}
\usepackage{amsmath,amssymb,amsthm,srcltx}
\usepackage{mathtools}
\usepackage[dvips]{color}
\usepackage[margin]{fixme}
\usepackage{enumerate}
\usepackage{stmaryrd}
\usepackage[colorinlistoftodos]{todonotes}
\usepackage{bbm}
\usepackage{dsfont}
\numberwithin{equation}{section}
\usepackage{fullpage}
\usepackage{subcaption}

\def\dbE{\mathbb{E}}
\def\dbF{\mathbb{F}}
\def\dbG{\mathbb{G}}

\def\dbM{\mathbb{M}}

\def\dbP{\mathbb{P}}
\def\dbR{\mathbb{R}}

\def\dbQ{\mathbb{Q}}

\newcommand{\cC}{\mathcal{C}}
\newcommand{\cD}{\mathcal{D}}
\newcommand{\cE}{\mathcal{E}}
\newcommand{\cF}{\mathcal{F}}
\newcommand{\cG}{\mathcal{G}}
\newcommand{\cH}{\mathcal{H}}
\newcommand{\cI}{\mathcal{I}}
\newcommand{\cK}{\mathcal{K}}
\newcommand{\cL}{\mathcal{L}}
\newcommand{\cM}{\mathcal{M}}

\newcommand{\cP}{\mathcal{P}}
\newcommand{\cQ}{\mathcal{Q}}

\newcommand{\cS}{\mathcal{S}}
\newcommand{\cT}{\mathcal{T}}
\newcommand{\cU}{\mathcal{U}}
\newcommand{\cV}{\mathcal{V}}

\newcommand{\cY}{\mathcal{Y}}
\newcommand{\cZ}{\mathcal{Z}}

\newcommand{\Cbf}{\mathbf{C}}
\newcommand{\frakC}{\mathfrak{C}}

\def\no{\noindent}

\def\q{\quad}

\def\e{\varepsilon}
\def\Om{\Omega}
\def\om{\omega}

\def\t {\tau}
\def\si {\sigma}

\def\we {\wedge}
\def\eps{\varepsilon}

\newcommand{\ba}{\begin{array}} 
\newcommand{\ea}{\end{array}}
\newcommand{\be}{\begin{equation}}
\newcommand{\ee}{\end{equation}}
\newcommand{\bea}{\begin{eqnarray}}
\newcommand{\eea}{\end{eqnarray}}
\newcommand{\beaa}{\begin{eqnarray*}}
\newcommand{\eeaa}{\end{eqnarray*}}

\newcommand{\Ind}{\mathbf{1}}

\DeclareMathOperator*{\esssup}{ess\,sup}
\DeclareMathOperator*{\essinf}{ess\,inf}

\newtheorem{theorem}{Theorem}[section] 
\newtheorem{assumption}[theorem]{Assumption}

\newtheorem{definition}[theorem]{Definition}

\newtheorem{proposition}[theorem]{Proposition}

\theoremstyle{definition}
\newtheorem{example}[theorem]{Example}
\newtheorem{remark}[theorem]{Remark}

\begin{document}

\title{Random horizon principal-agent problems
        \thanks{
        This work benefits from the financial support of the ERC Advanced Grant 321111, and the Chairs {\it Financial Risk} and {\it Finance and Sustainable Development}, and the NSFC Grant No.~11801365.
       }}

\author{Yiqing LIN\footnote{School of Mathematical Sciences, Shanghai Jiao Tong University, 200240 Shanghai, China.}
        \and Zhenjie REN\footnote{CEREMADE, Universit\'e Paris Dauphine, F-75775 Paris Cedex 16, France. }
        \and Nizar TOUZI\footnote{CMAP, \'Ecole Polytechnique, F-91128 Palaiseau Cedex, France.}
        \and Junjian YANG\footnote{FAM, Fakult\"at f\"ur Mathematik und Geoinformation, Vienna University of Technology, A-1040 Vienna, Austria.}
       }

\date{\today}

\maketitle

\begin{abstract} 
 We consider a general formulation of the random horizon principal-agent problem with a continuous payment and a lump-sum payment at termination. 
 In the European version of the problem, the random horizon is chosen solely by the principal with no other possible action from the agent than exerting effort on the dynamics of the output process. 
 We also consider the American version of the contract, where the agent can also quit by optimally choosing the termination time of the contract. 
 Our main result reduces such non-zero-sum stochastic differential games to appropriate stochastic control problems which may be solved by standard methods of stochastic control theory. 
 This reduction is obtained by following the Sannikov \cite{Sannikov08} approach, further developed in \cite{CPT18}. 
 We first introduce an appropriate class of contracts for which the agent's optimal effort is immediately characterized by the standard verification argument in stochastic control theory. 
 We then show that this class of contracts is dense in an appropriate sense, so that the optimization over this restricted family of contracts represents no loss of generality. 
 The result is obtained by using the recent well-posedness result of random horizon second-order backward SDE in \cite{LRTY20}.
\end{abstract}

\vspace{-0.2cm}\newline


\section{Introduction}

The principal-agent problem is a classical moral hazard problem in economics with many applications in corporate governance and industrial economics, which is formulated as a Stackelberg game. The principal (she) delegates the management of an output process to the agent (he). A contract is signed beforehand, stipulating the terms of an incentive payment. The agent devotes a costly effort for the management of the output. Then, given the contract offered by the principal, he returns an optimal effort response which best balances between his cost of effort and the proposed compensation. Finally, the principal chooses the optimal contract so as to to incite the agent's effort to serve her interest. A crucial feature of the problem is that the principal only observe the output process, and has no access to the amount of effort exerted by the agent.

There is tremendous literature on this topic, mainly in the one-period setting; we refer to the seminal book \cite{BD05}. The first continuous time formulation of this problem was introduced by Holmstr\"{o}m and Milgrom \cite{HM87}. The importance of the continuous time formulation was best illustrated by the simplicity of the results. Since then, there has been a stream of research in this direction using the technique of calculus of variations. We refer to the book by Cvitani\'c and Zhang \cite{CZ12} for the main achievements with this point of view. 

An original method was introduced by Sannikov \cite{Sannikov08} which exploits in a very clever way the agent dynamic value process. This method was related by Cvitani\'c, Possama\"{\i} and Touzi \cite{CPT18} to the theory of backward stochastic differential equations, and extended to the setting where the agent is allowed to control the diffusion of the out process. Such an extension is particularly relevant in portfolio management as illustrated in Cvitani\'c, Possama\"{\i} and Touzi \cite{CPT17}. We also refer to A\"{\i}d, Possama\"{\i} and Touzi \cite{APT19} for an application to the demand-response problem in electricity pricing.

Sannikov's approach consists of deriving a representation of the dynamic value process, by means of the dynamic programming principle, and then reformulating the principal objective as a control problem on the coefficients of this representation. By this methodology, the initial Stackelberg stochastic differential game is reduced to a stochastic control problem. Notice that this representation is nothing but the non-Markovian version of the Hamilton-Jacobi-Bellman equation corresponding to the agent problem. The extension to the controlled diffusion setting follows the same idea but requires, in addition, a density result of second order backward SDEs. 

The main objective of this paper is to extend the reduction result of \cite{CPT18} to the random horizon context. In particular, this allows one to cover the seminal paper of Sannikov \cite{Sannikov08}. The random horizon setting is commonly used in applications in order to reduce the dimensionality of control problems, as the time variable disappears in homogeneous formulations. Consequently, if the controlled state is one-dimensional, the HJB partial differential equation reduces to a nonlinear ordinary differential equation whose analysis is usually simpler, and which may be found in explicit form in several cases. 

We shall introduce two versions of the random horizon principal-agent problem. The first is a direct extension of the finite horizon one, and is named as the European contracting problem. The second one corresponds to the setting of \cite{Sannikov08}, and is named as the American contracting problem due to the possibility offered to the agent and the principal to terminate the contract at some chosen stopping time. In other words, both actors are faced with an optimal stopping problem in addition to optimally controlling the coefficients of the controlled output process.

As in \cite{CPT18}, our main results, both for the European and the American contracting problems, rely on a density property of second order backward SDEs in an appropriate family of solutions of the non-Markovian version of the agent HJB equation. The corresponding well-posedness result is obtained in our accompanying paper \cite{LRTY20}. However, while the density argument for the European contracting follows the corresponding argument in \cite{CPT18}, the American contracting argument requires a new justification based on understanding the principal choice of the optimal termination time of the contract, given the optimal stopping response of the agent.

This paper is organized as follows. The random horizon principal-agent problem is described in Section \ref{sec:formulation} both in its European and American formulations. Section \ref{sec:firstbest} shows that our European contracting problem does not coincide with the corresponding first best contracting problem in the context where the discount factors of both actors are deterministic. This is in contrast with the deterministic horizon situation. In Section \ref{sec:mainresult}, we state our main reduction results, and we report their proof based on a density property of second order backward SDEs. We illustrate the usefulness of our reduction result through a solvable example in Section \ref{sec:example}. Finally, Section \ref{sec:proof} contains the proof of the key density result.

\paragraph{Preliminaries and notation}

Given an integer $d$ and some initial condition $X_0\in\dbR^d$, we introduce the canonical space of continuous paths $\Omega:=\big\{\omega\in\cC\big(\dbR_+, \dbR^d\big):~\omega_0=X_0 \big\}$, 
   equipped with the distance defined by $ \|\omega-\omega'\|_\infty:=\sum_{n\ge 0}2^{-n}\big(\sup_{0\leq t\leq n}|\omega_{t}-\omega'_{t}|\wedge 1\big).$ We denote by $\mathfrak{M}_+^1(\Omega)$ the collection of all probability measures on $\Omega$.
 
The canonical process $X$ is defined by $X_t(\omega) := \omega_t$, for all $\omega\in\Omega$, with corresponding canonical filtration $\dbF=(\cF_t)_{t\geq 0}$. 
 We also introduce the  the right limit $\dbF^+=(\cF^+_t)_{t\geq 0}$  of $\dbF$, and for a measure $\dbP\in\mathfrak{M}_+^1(\Omega)$, the augmentation $\dbF^{+,\dbP}$ of the filtration $\dbF^+$ under $\dbP$. 
 For a subset $\cP\subseteq\mathfrak{M}_+^1(\Omega)$, we introduce $\dbF^\cP:=\big(\cF^\cP_t\big)_{t\geq 0}$ and $\dbF^{+,\cP}:=\big(\cF^{+,\cP}_t\big)_{t\geq 0}$, where
  $$ \cF_t^\cP:=\bigcap_{\dbP\in\cP}\cF_t^\dbP \quad \mbox{and} \quad \cF^{+,\cP}_t := \bigcap_{\dbP\in\cP}\cF^{+,\dbP}_t.  $$
 
 We say that a property holds $\cP$-quasi-surely, abbreviated as $\cP$-q.s., if it holds $\dbP$-a.s.~for all $\dbP\in\cP$. 
 The universal filtration $\dbF^U:=\big(\cF^U_t\big)_{t\geq 0}$ and the corresponding (right-continuous) completion $\dbF^{+,U}:=\big(\cF^{+,U}_t\big)_{t\geq 0}$ correspond to the case $\cP=\mathfrak{M}_+^1(\Omega)$.

 We denote by $\cP_{\mbox{\tiny loc}}\subseteq\mathfrak{M}_+^1(\Omega)$ the collection of probability measures $\dbP$ such that $X$ is a continuous $\dbP$-local martingale with quadratic variation process absolutely continuous in $t$, 
   with respect to the Lebesgue measure, with corresponding density
   \begin{equation*}
     \widehat{\sigma}^2_t := \limsup_{n\rightarrow \infty} n \big(\langle X, X\rangle_t-\langle X, X\rangle_{(t-\frac{1}{n})\vee 0}\big), \quad t>0.
   \end{equation*}
 Here, the quadratic covariation process $\langle X \rangle$ is pathwisely well-defined by Karandikar \cite{Kar95}.  
 Then, for all $\dbP\in \cP_{\mbox{\tiny loc}}$, we may find a Brownian motion $W$ such that
   \begin{equation*}
     X_t = \int^t_0\widehat{\sigma}_sdW_s, \quad t\geq 0, \quad \dbP\mbox{-a.s.}
   \end{equation*}
For a stopping time $\tau$ we define the stochastic interval
  $\llbracket 0, \tau\rrbracket:=\{(t,\omega)\in\dbR_+\times\Omega:t\leq\tau(\omega)\}$.

\vspace{3mm}

We next enlarge the canonical space to $\overline \Om:=\Om\times\Om$ and denote by $(X,W)$ the coordinate process in $\overline \Omega$. 
Denote by $\overline\dbF$ the filtration generated by $(X,W)$. 
For each $\mathbb{P}\in  \cP_{\mbox{\tiny loc}}$ we may construct a probability measure $\overline\dbP$ on $\overline\Omega$ such that $\overline\dbP\circ X^{-1}=\dbP$, 
    $W$ is a $\overline\dbP$-Brownian motion, and $dX_t = \widehat\si_t dW_t$, $\overline\dbP$-a.s. 
From now on, we abuse notation, and keep using $\dbP$ to represent $\overline\dbP$ on $\overline \Om$. 
Denote by $\cQ_L(\mathbb{P})$ the set of all probability measures $\dbQ^\lambda$ such that
  \begin{eqnarray}\label{DQP} 
    {\rm D}_t^{\dbQ^\lambda|\dbP} := \frac{d\dbQ^\lambda}{d\dbP}\bigg|_{\overline\cF_t} = \exp{\bigg(\int_0^t\lambda_s\cdot dW_s - \frac{1}{2}\int_0^t|\lambda_s|^2ds\bigg)}, \quad t\geq 0 
  \end{eqnarray}
  for some $\dbF^{+, \mathbb{P}}$-progressively measurable process $\lambda=(\lambda)_{t\geq 0}$ uniformly bounded by $L$. 
By Girsanov's theorem,
  $W^{\lambda}:= W-\int_0^{\cdot}\lambda_s ds $ is a $\dbQ^\lambda$-Brownian motion on any finite horizon, and thus $X^{\lambda}:= X -\int_0^\cdot \widehat\si_t \lambda_t dt $ is a $\dbQ^\lambda$-martingale on any finite horizon. We denote 
   \begin{align*}
   	\cE^\dbP[\cdot] := \sup_{\dbQ\in\cQ_L(\dbP)}\dbE^{\dbQ}[\cdot]~~\mbox{for}~\dbP\in\cP_{\mbox{\tiny loc}}~~\mbox{and}~~\cE^\cP[\cdot] := \sup_{\dbP\in\cP}\cE^{\dbP}[\cdot]~\mbox{for a subset}~
   	\cP\subseteq \cP_{\mbox{\tiny loc}}.
   \end{align*}
Let $p>1$ and $\alpha\in\dbR$, and let $\tau$ be an $\dbF^{+,\dbP}$-stopping time. 
Let $\dbG:=\{\cG_t\}_{t\geq 0}$ be a filtration with $\cG_t\supseteq\cF_t$ for all $t\geq 0$, so that $\tau$ is also a $\dbG$-stopping time. 
We denote the following: 
\begin{itemize}
 \item $\cL^p_{\alpha,\tau}(\cP, \dbG)$, the space of $\dbR$-valued, $\cG_\tau$-measurable $\dbR$-valued random variables $\xi$, such that 
          $$ \|\xi\|^p_{\cL_{\alpha,\t}^p(\cP)} := \cE^\cP\big[\big|e^{\alpha\tau}\xi\big|^p\big] <\infty. $$
 \item $\cD^p_{\alpha, \tau}(\cP, \dbG)$, the space of scalar c\`adl\`ag $\dbG$-adapted processes $Y$ such that
           $$ \|Y\|^p_{\cD^p_{\alpha,\tau}} 
               := 
              \cE^{\cP}\left[\sup_{0\leq t\leq\tau}\big|e^{\alpha t}Y_t\big|^p\right] <\infty. $$
 \item $\cH^p_{\alpha, \tau}(\cP, \dbG)$, the space of $\dbR^d$-valued, $\dbF^{+}$-progressively measurable processes $Z$ such that
          $$ \|Z\|^p_{\cH^p_{\alpha,\tau}} 
               := 
              \cE^{\cP} \left[\left(\int_0^\tau \big|e^{\alpha t}\widehat\sigma_t^{\top}Z_t\big|^2dt\right)^{\frac{p}{2}}\right] <\infty. $$
\end{itemize}

\section{Principal-agent problem} 
\label{sec:formulation}

\subsection{Controlled state equation}
The agent's effort $\nu=(\alpha,\beta)$ is an $\dbF$-optional process with values in $A\times B$ for some subsets $A$ and $B$ of finite-dimensional spaces. We denote the set of such effort processes as $\mathfrak{U}$. 
The output process takes values in $\dbR^d$, with distribution defined by means of the controlled coefficients:
 \begin{align*}
   & \lambda:\dbR_+\times\Omega\times A\longrightarrow \dbR^d, ~\mbox{bounded, }~\lambda(\cdot,a)~\dbF\mbox{-optional for any}~a\in A, \\
   &  \sigma:\dbR_+\times\Omega\times B\longrightarrow \mathcal M_{d}(\dbR), ~\mbox{bounded, }~\sigma(\cdot,b)~\dbF\mbox{-optional for any}~b\in B,
 \end{align*} 
where $\mathcal M_{d}(\mathbb R)$ denotes the space of all square $d\times d$ matrices with real entries.
The controlled state equation is defined by the SDE:
 \begin{equation} \label{Xalphabeta}
   X_t = X_0 + \int_0^t\sigma_r(X,\beta_r)\big(\lambda_r(X,\alpha_r)dr + dW_r\big), \quad t\geq 0, 
 \end{equation}
 where $W$ is a $d$-dimensional Brownian motion. 
Notice that the processes $\alpha$ and $\beta$ are functions of the path of $X$. As is standard in probability theory, the dependence on the canonical process will be suppressed.

A {\it control model} is a weak solution of \eqref{Xalphabeta} defined as a pair $\dbM:=(\dbP,\nu)\in\mathfrak{M}_+^1(\Omega)\times\mathfrak{U}$. 
We denote $\cM$ to be the collection of all such control models, as opposed to control processes. 
We assume throughout this paper the following implicit condition on $\sigma$ (see Remark \ref{rem:driftlessSDE} below):
 \begin{equation}\label{eq:Mnonvide}
   \cM\neq\emptyset.
 \end{equation}
This condition is satisfied, for instance, if $x\mapsto\sigma_t(x,b)$ is bounded and continuous for some constant control $b\in B$, see, e.g., \cite[Theorem 5.4.22, Remark 5.4.23]{KS91}.

Notice that we do not restrict the controls to those for which weak uniqueness holds. 
Moreover, by Girsanov's theorem, two weak solutions of \eqref{Xalphabeta} associated with $(\alpha,\beta)$ and $(\alpha',\beta)$ are equivalent. 
However, different diffusion coefficients induce mutually singular weak solutions of the corresponding stochastic differential equations.

We finally introduce the following sets:
 \begin{equation*}
 \begin{array}{ll}
   \cP(\nu):=\big\{\dbP\in\mathfrak{M}_+^1(\Omega),\; (\dbP,\nu)\in\cM\big\}, &  \displaystyle\cP:=\cup_{\nu\in\mathfrak{U}}\cP(\nu), \\
   \cU(\dbP):=\big\{\nu\in\mathfrak{U},\; (\dbP,\nu)\in\cM\big\},          &  \displaystyle\cU:=\cup_{\dbP\in\mathfrak{M}_+^1(\Omega)}\cU(\dbP).
 \end{array}
 \end{equation*}

\begin{remark} \label{rem:driftlessSDE}
 By the boundedness of $\lambda$, we may connect any admissible model $(\dbP,\nu)\in\cM$ to a subset of $\cP_b$ as follows. 
 Let $(\dbQ,\beta)$ be an arbitrary weak solution of the driftless SDE
  \begin{equation} \label{SDEsigma}
    X_t = X_0+\int_0^t\sigma_r(X,\beta_r)dW_r, \quad t\geq 0,
  \end{equation}
for some optional $B$-valued process $\beta$. 
 Then, $\dbQ\in\cP_{\rm loc}$, and we may use the Girsanov change of measure theorem to define for every $A$-valued optional process $\alpha$ a pair $\dbM:=\big(\dbP,(\alpha,\beta)\big)$ which solves the SDE \eqref{Xalphabeta} by setting 
   $$
  \frac{d\dbP}{d\dbQ}\bigg|_{\overline\cF_t} 
  = 
  \exp\left(\int_0^t\lambda_s(X,\alpha_s)\cdot dW_s - \frac{1}{2}\int_0^t|\lambda_s(X,\alpha_s)|^2ds\right),\quad t\geq 0. $$
  
Conversely, any admissible model $\dbM=\big(\dbP,(\alpha,\beta)\big)\in\cM$ induces a probability measure $\dbQ\in\cP_{\rm loc}$ by the last Girsanov equivalent change of measure.  \qed
\end{remark}

\subsection{Agent's problem}
The effort exerted by the agent is costly, with cost of effort measured by the function
 \begin{align*}
 	& c:\dbR_+\times\Omega\times A \times  B \rightarrow \dbR_+, \mbox{ measurable, } \\
 	& c(\cdot,u)~\dbF\mbox{-optional for all}~u\in A\times B,
 	\mbox{ and } c(.,0)=0. 
 \end{align*}
Let $(\dbP,\nu)\in\mathcal M$ be fixed. The canonical process $X$ is called the {\it output} process, and the control $\nu$ is called the agent's {\it effort} or {\it action}. 
The agent exerts the effort process $\nu$ to control the (distribution of the) output process defined by the state equation \eqref{Xalphabeta}, while subject to cost of effort at rate $c(X,\nu)$. 
The agent values future income through the discount factor $\cK^{\nu}:=e^{-\int_0^. k_r(\nu_r)dr}$, where
  $$ k:\dbR_+\times\Omega\times A\times B\rightarrow \dbR \mbox{ bounded, with } k(\cdot,u)~\dbF\mbox{-optional for all}~u\in A\times B, $$
 and $k(.,0)=k_0$ for some constant $k_0>0$.

\vspace{3mm} 
 
{\it A contract} is a triple $\Cbf=(\tau^{\scriptscriptstyle\rm P},\pi,\xi)$ composed of the following:
  \begin{itemize}
   \item a finite stopping time $\tau^{\scriptscriptstyle \rm P}$, representing the termination time of the contract;
   \item an optional process $\big(\pi_{t\wedge\tau^{\scriptscriptstyle\rm P}}\big)_{t\ge 0}$, representing a rate of payment from the principal to the agent; and
   \item an $\cF_{\tau^{\scriptscriptstyle\rm P}}$-measurable random variable $\xi$, representing the final compensation at retirement.
  \end{itemize}

The principal observes only the output process $X$ and has no access to the information on the agent's effort. 
Consequently, the components of the contract $\Cbf$ can only be contingent on $X$, which is immediately encoded in our weak formulation setting.

The set of {\it admissible contracts} $\frakC$ consists of all such contracts which satisfy in addition the technical requirements reported in subsection \ref{sect:admissiblecontracts} below.

The agent's preferences are defined by a continuous strictly increasing utility function $U:\dbR\rightarrow\dbR$. 
Given a contract $\Cbf=(\tau^{\scriptscriptstyle\rm P},\pi,\xi)$, we shall consider in this paper two possible contracting problems which are both relevant in the economics literature.

\vspace{3mm}
 
\no {\bf Agent cannot quit.} We first consider the contract problem as in Sannikov \cite{Sannikov08}. By analogy with derivatives securities, we refer to this setting as that of an {\it European contracting problem}. 
The objective function is defined by
 \begin{equation}\label{JE}
   J^{\rm E}(\dbM,\Cbf)
   :=  \dbE^{\dbP}\bigg[\cK^{\nu}_{\tau^{\scriptscriptstyle\rm P}} U_{\rm A}(\xi)
                                    +\int_0^{\tau^{\scriptscriptstyle\rm P}} \mathcal K^{\nu}_t \big(U_{\rm A}(\pi_t)-c_t(\nu_t)\big)dt \bigg]
   \,\,\,\mbox{for all}\,\,\,
   \dbM=(\dbP,\nu)\in\cM.
 \end{equation}
Throughout this paper, we adopt the convention $\infty-\infty=-\infty$, implying that the above expectation $J^{\rm E}(\dbM,\Cbf)$ is well-defined. 
The European agent aims at optimally choosing the effort, given the promised compensation contract $\Cbf$:
 \beaa\label{EuroAgent}
    V^{\rm E}(\Cbf) := \sup_{\dbM\in\cM} J^{\rm E}(\dbM,\Cbf), \quad \Cbf\in\mathfrak{C},
 \eeaa
with the convention $\sup\emptyset=-\infty$, which also prevails throughout this paper. 

A control model $\widehat{\dbM}=(\widehat{\dbP},\widehat{\nu})\in\cM$ is an {\it optimal response} to contract $\Cbf$ if $V^{\rm E}(\Cbf)=J^{\rm E}\big(\widehat{\dbM},\Cbf\big)$. 
We denote by $\widehat{\cM}^{\rm E}(\Cbf)$ the (possibly empty) set of all such optimal control models.

\vspace{3mm}

\no {\bf Agent can quit.} 
We now introduce a new setting which we name as that of the {\it American contracting problem}. We assume that the agent may chose a retirement time $\tau$ before the contract terminates.  
After retirement, the agent receives no more transfers from the principal, i.e., $\xi=0$ and $\pi=0$ on $\{t\ge\t\wedge\tau^{\scriptscriptstyle\rm P}\}$. 
As $c_t(0)=0$ and $k_t(0)=k_0$, the (dynamic) value function of the agent at retirement is given by 
 \beaa
 \int_{\t\wedge\t^{\scriptscriptstyle\rm P}}^\infty e^{-k_0(t-\t\wedge\t^{\scriptscriptstyle\rm P})} U_{\rm A}(0)dt
   =
 \frac{U_{\rm A}(0)}{k_0}
   =:
 U_{\rm A}(\rho).
 \eeaa
Given this definition of the constant $\rho$, we denote by $\frakC^{\rm A}$ the collection of all pairs $\Cbf^{\rm A}=(\tau^{\scriptscriptstyle\rm P},\pi)$ such that $\Cbf:=(\Cbf^{\rm A},\rho)\in\frakC$. 
We denote by $\cM^{\rm A}$ the collection of all decision variables $(\tau,\dbM)$ for the agent, where $\t$ is an $\dbF$-stopping time, and $\dbM=(\dbP,\nu)\in\cM$. 
The American agent's objective function is defined by
 \begin{equation*}
   J^{\rm A}(\t, \dbM, \Cbf^{\rm A})
    :=
   \dbE^{\dbP}\bigg[\cK^{\nu}_{\tau\we\t^{\scriptscriptstyle\rm P}} U_{\rm A}(\rho)
                               +\int_0^{\tau\we\t^{\scriptscriptstyle\rm P}} \mathcal K^{\nu}_t \big(U_{\rm A}(\pi_t)-c_t(\nu_t)\big)dt \bigg]
                               \quad\mbox{for all}~(\t,\dbM)\in\cM^{\rm A},
 \end{equation*}
and aims at optimally choosing the effort and the quitting time, given the promised compensation contract $\Cbf^{\rm A}$:
 \bea
 V^{\rm A}\big(\Cbf^{\rm A}\big) 
  :=
 \sup_{(\t,\dbM)\in\cM^{\rm A}} J^{\rm A}\big(\t, \dbM,\Cbf^{\rm A}\big), 
 \quad \Cbf^{\rm A}\in\frakC^{\rm A}.
 \eea
We say that $\big(\widehat\t,\widehat{\dbM}\big)\in\cM^{\rm A}$ is an optimal response to contract $\Cbf^{\rm A}$ if $V^{\rm A}(\Cbf^{\rm A})=J^{\rm A}\big(\widehat\t,\widehat{\dbM},\Cbf^{\rm A}\big)$. 
We denote by $\widehat{\cM}^{\rm A}(\Cbf^{\rm A})$ the (possibly empty) set of all such optimal responses.

\subsection{Principal's problem}
The contracts which can be offered by the principal are those admissible contracts which are subject to the additional restriction:
 \begin{align} \label{def:FrakC}
 	 \mathfrak{C}^{\rm a}_R := \big\{ \Cbf\in\frakC^a : V^a(\Cbf)\geq U_{\rm A}(R)\big\}, 
 	  \quad {\rm a}\in\{{\rm E},{\rm A}\},
 \end{align}
where $\frakC^{\rm E}:=\frakC$, and $R$ is a given participation threshold representing the minimum satisfaction level required by the agent in order to accept the contract.

The principal benefits from the value of the output $X$ and pays the agent as promised in the contract $\Cbf$, namely, she pays a continuous compensation at the rate $\pi$, and 
\begin{itemize}
 \item in case of a European contract, a final compensation $\xi$ at the termination time $\tau^{\scriptscriptstyle\rm P}$,
 \item in case of an American agent, $\xi=0$ at the agent quitting time $\widehat\tau\wedge\tau^{\scriptscriptstyle\rm P}$. 
\end{itemize}

\no This leads to the following definitions for the second-best principal's problem under European and American contracts, respectively: 
 \begin{align*}
 	V^{\rm PE} &:= \sup_{\Cbf\in\mathfrak{C}^{\rm E}_R} 
 	\sup_{\dbM\in\widehat{\cM}^{\rm E}(\Cbf)}
 	J^{\rm P}(\dbM,\Cbf), \\
 	V^{\rm PA} &:= \sup_{(\tau^{\scriptscriptstyle\rm P},\pi)\in\mathfrak{C}^{\rm A}_R} 
 	\sup_{(\tau,\dbM)\in\widehat{\cM}^{\rm A}(\tau^{\scriptscriptstyle\rm P},\pi)}
 	J^{\rm P}(\dbM,\tau\wedge\tau^{\scriptscriptstyle\rm P},\pi,0),
 \end{align*}
where, for all $\Cbf=(\tau,\pi,\xi)$:
 \begin{align*}
 	J^{\rm P}(\dbM,\Cbf) 
 	:=
 	\dbE^{\dbP}\bigg[\cK^{\rm P}_\tau U_{\rm P}\big(\ell_\tau-\xi\big) + \int_0^\tau \cK^{\rm P}_rU_{\rm P}(-\pi_r)dr\bigg].
 \end{align*}
Here, $U_{\rm P}:\dbR\rightarrow \dbR$ is a given nondecreasing utility function, $\ell:\Omega\rightarrow \dbR$ is a liquidation function with linear growth, $\ell_\tau:=\ell(X_{.\wedge\tau})$, and $\cK^{\rm P}_t:=e^{-\int_0^t k^{\rm P}_rdr}$, $t\geq 0$, is a discount factor, defined by means of a discount rate function 
   $$ k^{\rm P}:\dbR_+\times\Omega\rightarrow \dbR~~\mbox{ bounded and $\dbF$-optional}. $$
By our convention $\sup\emptyset=-\infty$, notice that the principal only offers those admissible contracts which induce a nonempty set of optimal responses, i.e. $\widehat{\cM}^{\rm E}(\Cbf)\neq\emptyset$ in the European case and $\widehat{\cM}^{\rm A}(\tau^{\rm P},\pi)\neq\emptyset$ in the American case. We also observe that, following the standard economic convention, the above definition of the principal's criterion assumes that, in the case where the agent is indifferent between various optimal responses, he implements the one that is the best for the principal. 
  

\begin{remark}
Careful readers may have noted that in this paper we analyze both the cases where the agent can or cannot quit the contract. However, we always allow the principal to end the contract at a chosen time $\t$. What would happen if the principal has no right to end the contract before the maturity? In fact,  the principal can always induce the agent to quit the contract at a stopping time $\t$ chosen by the principal (without ending the contract herself), by offering low payment at any time $t\neq \tau$. We refer the interested readers to Remark 2.2 in Cvitani\'c, Wan and Zhang \cite{CWJ08} for detailed discussion. 
\end{remark}

\subsection{Admissible contracts}
\label{sect:admissiblecontracts}

We now provide the precise definition of the set of admissible contracts $\frakC$.
We need the following additional notation: 
  $$ \Omega_t^\omega:= \left\{\omega'\in\Omega : \omega|_{[0,t]} = \omega'|_{[0,t]}\right\},
  ~~
  (\om\otimes_t\om')_s:= \Ind_{\{s\le t\}}\om_s+\Ind_{\{s>t\}}(\om_t+\om'_{s-t}), 
  $$
  and 
  $$ \xi^{t,\om}(\om'):=\xi(\om\otimes_t\om'),\quad X^{t,\om}_s(\om'):=X_{t+s}(\om\otimes_t\om'),\quad \vec{\tau}^{t,\om}:=\tau^{t,\om}-t. 
  $$
We also introduce the dynamic version of $\cP$ by considering the controlled SDE on $[t,\infty)$ issued from the path $\omega\in\Omega$:
  $$ 
  \cP(t,\omega)
  :=
  \Big\{\dbP\in \mathfrak{M}_+^1(\Omega_t^\omega): dX^{t,\omega}_s = \sigma^{t,\omega}_s(X^{t,\omega},\beta_s)\big(\lambda^{t,\omega}_s(X^{t,\omega},\alpha_s)ds + dW_s\big), \dbP\mbox{-a.s.}, (\alpha,\beta)\in\mathfrak{U} \Big\}. 
  $$
In particular $\cP=\cP(0,\mathbf{0})$. We shall use the nonlinear expectations $ \cE^{\cP(t,\omega)}[\cdot] := \sup_{\dbP\in\cP(t,\omega)}\cE^{\dbP}[\cdot].$

\begin{definition} \label{assum:xi.pi}
 \begin{enumerate}[{\rm (i)}]
  \item We denote by $\tau\in\cT$ the collection of all stopping times $\tau$ satisfying 
          \begin{equation} \label{eq:tau}
             \lim_{n\to\infty}\cE^{\cP}\big[\Ind_{\{\tau\geq n\}}\big]=0. 
          \end{equation}
  \item An admissible contract is a triple $\Cbf=(\tau,\pi,\xi)$, with $\tau\in\cT$, and
   \begin{equation} \label{contract-growth}
      \cE^{\cP(t,\omega)}\big[\big|e^{\rho\vec{\tau}^{t,\omega}}U_{\rm A}(\xi^{t,\omega})\big|^q\big] + \cE^{\cP(t,\omega)}\bigg[\bigg(\int_0^{\vec{\tau}^{t,\omega}} \big|e^{\rho r}U_{\rm A}(\pi_r^{t,\omega})\big|^2 dr\bigg)^{\frac{q}{2}}\bigg]<\infty,
   \end{equation} 
      for some $q>1$ and $\rho>-\mu$, where 
          $$ \mu = \inf_{u\in A\times B}\inf_{t\geq 0}\essinf_{\omega\in\Omega}k_t(\omega,u). $$
     We denote by $\frakC$ the set of admissible contracts. 
 \end{enumerate}
\end{definition}

The following condition ensures that $J^{\rm E}$ and $J^{\rm A}$ are finite for each contract $\Cbf\in\frakC$.

\begin{assumption} \label{assum:costc}
The cost function $c$ is bounded by $\overline c$ satisfying, for some $\rho>-\mu$ and $q>1$,
 \begin{equation} \label{eq:costc}
  \cE^{\cP(t,\omega)}\bigg[\bigg(\int_0^{\vec{\tau}^{t,\om}} 
                                                 \big|e^{\rho r}\,\overline c_r^{t,\omega}\big|^2 dr
                                         \bigg)^{\frac{q}{2}}\bigg] 
    < \infty,
  ~~\mbox{for all}~~(t,\omega)\in\llbracket 0,\tau\rrbracket,~\mbox{and}~\tau\in\cT.
 \end{equation}
\end{assumption}

\begin{remark}
 For $(t,\omega)=(0,\mathbf{0})$, we have $\cP(t,\omega)=\cP$ and 
  \begin{align*}
  	\cE^{\cP}\bigg[\bigg(\int_0^{\tau} \big|e^{\rho r}\overline c_r\big|^2 dr\bigg)^{\frac{q}{2}}\bigg] 
  	+ \cE^{\cP}\big[|e^{\rho\tau}U_{\rm A}(\xi)|^q\big] 
  	+ \cE^{\cP}\bigg[\bigg(\int_0^\tau |e^{\rho r}U_{\rm A}(\pi_r)|^2 dr\bigg)^{\frac{q}{2}}\bigg]
  	< \infty.
  \end{align*}
\end{remark}

\section{Comparison with first best contracts}
\label{sec:firstbest}

 We first observe that our reduction result of the European contracting problem extends to the case where the set of admissible contracts $\frakC$ is replaced by $\frakC[\mathfrak{T}]:=\{\mathbf{C}=(\tau^{\rm P},\pi,\xi)\in\frakC^{\rm E}_R: \tau^{\rm P}\in\mathfrak{T}\}$, for some subset $\mathfrak{T}$ of finite stopping times.  

In the economics literature, it is well-known that in the risk-neutral agent setting with deterministic maturity $\tau^{\rm P}=T$, the European principal-agent problem reduces to one single optimization problem corresponding to the case where the principal imposes the amount of effort that the agent devotes. 
This is the so-called first-best optimal contract problem where the principal has full power to choose both the contract and the agent effort. 
Under the European contracting rule, the first best risk sharing problem is defined by
 \bea
 V^{\rm PE}_{\rm fb} 
  :=
 \sup_{\substack{(\Cbf,\dbM)\in\mathfrak{C}[\mathfrak{T}]\times\cM \\ J^{\rm A}(\dbM,\Cbf)\geq R}}
                                           \dbE^{\dbP}\bigg[\cK^{\rm P}_\tau U_{\rm P}\big(\ell_\tau-\xi\big) 
                                           + \int_0^\tau \cK^{\rm P}_tU_{\rm P}(-\pi_t) dt\bigg].
 \eea
It is clear that $V^{\rm PE}_{\rm fb}\ge V^{\rm PE}$. Example \ref{ex:fb=sb} below shows that when the termination time $\tau^{\rm P}$ is not deterministic, the first and second best contracting problems do not coincide, in general. In this section, we provide precise conditions under which equality holds. We first need to assume that the agent's discount rate $k$ is independent of the effort, so that the agent discount factor is independent of the effort process: 
 \bea\label{eta}
  \cK_t:=\cK^\nu_t
    &\mbox{is independent of $\nu$, and we denote}&
        \eta_t:=\frac{\cK_t}{\cK^{\rm P}_t},
           ~~t\in [0,T].
 \eea
This condition is necessary in order to identify directly the optimal first best compensations $(\xi,\pi)$ of the principal independently of the agent's effort.

We also assume that the principal's utility function 
 \bea\label{UP}
 U_{\rm P}
 ~\mbox{is}~C^1,
 ~\mbox{increasing, and strictly concave, with}~
 ~U_{\rm P}'(-\infty)=\infty,
 ~U_{\rm P}'(\infty)=0,
 \eea 
and we introduce the corresponding convex conjugate
 $$
 U_{\rm P}^*(y)
 :=
 \sup_{x\in\dbR}\{U_{\rm P}(x)-xy\}
 =
 U_{\rm P}\circ(U'_{\rm P})^{-1}(y)-y(U'_{\rm P})^{-1}(y).
 $$
Finally, we shall denote for any function $F:\dbR_+\longrightarrow\dbR$ with appropriate measurability:
 \beaa
 \mathfrak{J}^F_\tau(\lambda)
   :=
 \cK^{\rm P}_\tau F(\lambda\eta_\tau)
 +\int_0^\tau\cK^{\rm P}_tF(\lambda\eta_t)dt,
 &\lambda\ge 0.&
 \eeaa

\begin{proposition} 
Consider a risk neutral agent, i.e., $U_{\rm A}={\rm Id}_\dbR$, with discount factor satisfying \eqref{eta}, and a principal with utility function satisfying \eqref{UP}. Let $\xi^{\lambda}:=\ell_\tau-(U_{\rm P}')^{-1}(\lambda\eta_\tau)$, $\pi^{\lambda}:=-(U_{\rm P}')^{-1}(\lambda\eta)$, and assume
   \beaa
    \hspace{14mm}
   ({\rm C}1) \hspace{10mm}
    \begin{array}{c}
    \mbox{for all}~\lambda\geq 0,~
    \mbox{the problem}
    ~\displaystyle
          v_{\rm fb}(\lambda)
              := \!\!\sup_{\substack{\tau\in\cT  \\  \dbM\in\cM(\tau,\pi^{\lambda},\xi^{\lambda})}} \!\!
                     \dbE^{\dbP}\left[\mathfrak{J}^{U_{\rm P}^*}_\tau(\lambda) + \lambda \mathfrak{H}_\tau(\nu) \right]
    \\ \displaystyle \mbox{has a solution $(\tau^{\lambda},\dbM^{\lambda})$, with}~
            \mathfrak{H}_\tau(\nu)
               := 
              \cK_\tau\ell_\tau- \!\int_0^\tau \!\cK_t c_t(\nu_t)dt -\!R.
          \end{array}
  \eeaa
\begin{enumerate}
\item[{\rm (1)}] Then, assuming, in addition, that
          \beaa
          ({\rm C}2) \hspace{10mm}
            0
             =
           \dbE^{\dbP^{\widehat\lambda}}
               \Big[\mathfrak{J}^{U_{\rm P}^*}_{\tau^{\widehat\lambda}}\big(\widehat\lambda\big)
                      -\mathfrak{J}^{U_{\rm P}\circ(U_{\rm P}')^{-1}}_{\tau^{\widehat\lambda}}\big(\widehat\lambda\big)
                      + \widehat\lambda \mathfrak{H}_{\tau^{\widehat\lambda}}\big(\nu^{\widehat\lambda}\big)
               \Big],
               &\mbox{for some}&
               \widehat\lambda>0, \hspace{10mm}
          \eeaa
\begin{enumerate}
 \item[{\rm (1-i)}] we have $V^{\rm PE}_{\rm fb}=v_{\rm fb}\big(\widehat\lambda\big)$, with optimal contract and effort 
         $$ \widehat{\mathbf{C}}:=\big(\widehat\tau^{\scriptscriptstyle\rm P},\widehat\pi,\widehat\xi\, \big):=\big(\tau^{\widehat\lambda},\pi^{\widehat\lambda},\xi^{\widehat\lambda}\big)
                \quad \mbox{and} \quad 
            \widehat\dbM:=\dbM^{\widehat\lambda}=\big(\dbP^{\widehat\lambda},\nu^{\widehat\lambda}\big).
            $$
 \item[{\rm (1-ii)}] $V^{\rm PE}=V^{\rm PE}_{\rm fb}$ if and only if $\widehat\dbM$ is also a solution of the problem
          \beaa
             \widehat v
              :=
             \sup_{\dbM\in\cM(\mathbf{\widehat C})}
                \dbE^{\dbP}
               \Big[\mathfrak{J}^{U_{\rm P}^*-U_{\rm P}\circ(U_{\rm P}')^{-1}}_{\tau^{\widehat\lambda}}(\widehat\lambda)
                    + \widehat\lambda \mathfrak{H}_{\tau^{\widehat\lambda}}(\nu)
               \Big];
          \eeaa
           in this case $\big(\widehat\tau^{\scriptscriptstyle\rm P},\widehat\pi,\widehat\xi\,\big)$ is also a second best optimal contract with optimal effort $\widehat\dbM$.
\end{enumerate}
\item[{\rm (2)}] Let $\mathfrak{T}=\{T\}$ be some fixed deterministic maturity, and let $\cK,\cK^{\rm P}$ be deterministic functions. Then, condition {\rm (C2)} is satisfied, and the problems $v_{\rm fb}$ and $\widehat v$ have the same set of solutions. 
       Consequently, $V^{\rm PE}=V^{\rm PE}_{\rm fb}$, and the first best and second best optimal contracting problems have the same solution.
\end{enumerate}
\end{proposition}

\begin{proof}
  (1-i) \, Let $\mathbf{C}=(\tau,\pi,\xi)$ and $\dbM=(\dbP,\nu)$ satisfy the participation constraint 
    $$ J^{\rm A}(\dbM,\mathbf{C})- R= \dbE^\dbP\left[\cK_{\tau}\xi + \ \int_0^{\tau}\cK_t\big(\pi_t-c_t(\nu_t)\big)dt\right] - R
  \ge 0. $$ 
  As $\widehat\lambda\ge 0$, as defined in condition (C2), we have,
  \allowdisplaybreaks
   \begin{align}
     J^{\rm P}(\dbM,\mathbf{C}) 
         &\leq \dbE^\dbP\bigg[\cK^{\rm P}_{\tau} U_{\rm P}(\ell_\tau-\xi) 
                                            +\!\! \int_0^{\tau}\!\!\!\!\cK^{\rm P}_t U_{\rm P}(-\pi_t)dt 
         + \widehat\lambda\left(\cK_{\tau}\xi 
                                             +\!\! \int_0^{\tau}\!\!\!\!\cK_t\big(\pi_t-c_t(\nu_t)\big)dt - R\right)\bigg] 
       \label{ineq:1} \\
         &= \dbE^\dbP\bigg[\cK^{\rm P}_{\tau}\big\{ U_{\rm P}(\ell_\tau-\xi) - \widehat\lambda\eta_{\tau}(\ell_\tau-\xi)\big\} + \int_0^{\tau}\!\!\cK^{\rm P}_t \big\{ U_{\rm P}(-\pi_t) - \widehat\lambda\eta_t(-\pi_t)\big\}dt \nonumber \\
         &\hspace{73mm} +  \widehat\lambda \bigg(\cK_{\tau}\ell_\tau- \int_0^{\tau}\!\!\cK_t c_t(\nu_t)dt - R\bigg) \bigg] \nonumber \\
         &\leq \dbE^\dbP\bigg[\cK^{\rm P}_{\tau} U_{\rm P}^*\big(\widehat\lambda\eta_{\tau}\big) + \int_0^\tau\cK^{\rm P}_t U_{\rm P}^*\big(\widehat\lambda\eta_t\big)dt + \widehat\lambda\mathfrak{H}_\tau(\nu)\bigg] \label{ineq:2} \\
         &= \dbE^\dbP\Big[\mathfrak{J}^{U^*_{\rm P}}_\tau\big(\widehat\lambda\big) + \widehat\lambda \mathfrak{H}_\tau(\nu)\Big] \leq v_{\rm fb}\big(\widehat\lambda\big),  \label{ineq:3}
   \end{align}
where the last inequality follows from the definition of $U^*_{\rm P}$.
By the arbitrariness of $(\tau,\pi,\xi,\dbM)$, this implies that $v_{\rm fb}\big(\widehat\lambda\big)$ defines an upper bound for $V^{\rm PE}_{\rm fb}$. 
Clearly, the contract $\mathbf{\widehat C}=\big(\widehat\tau^{\scriptscriptstyle\rm P},\widehat\pi,\widehat\xi\,\big)$ with effort $\widehat\dbM$, whose existence is guaranteed by condition (C1), restores the equality both in \eqref{ineq:2} and in \eqref{ineq:3}. 
By direct verification, we also see that the choice of $\lambda$ by means of condition (C2) restores equality in \eqref{ineq:1}. 
Hence, the last upper bound is achieved, and therefore $\big(\widehat\tau^{\scriptscriptstyle\rm P},\widehat\pi,\widehat\xi,\widehat\dbM\big)$ is a solution of the first best problem. 

\vspace{4mm}

\noindent (1-ii) \, The inequality $V_{\rm fb}^{\rm PE}\ge V^{\rm PE}$ is obvious. 
 In order to show the equality, a necessary and sufficient condition is that the optimal agent response $\widehat\dbM^{\mathbf{\widehat C}}=\widehat\dbM$, 
     i.e.~the agent's optimal response to the first best optimal contract coincides with the first best optimal effort. 

In order to complete the proof, we now show that given the contract $\big(\widehat\tau^{\scriptscriptstyle\rm P},\widehat\pi,\widehat\xi\,\big)$, the effort $\widehat\dbM=\big(\widehat\dbP,\widehat\nu\big)$ is an optimal response for the agent problem. 
Indeed, we directly compute that
 \begin{align*}
   J^{\rm E}\big(\dbM,\mathbf{\widehat C}\big) 
     &= \dbE^\dbP\left[\cK_{\widehat\tau^{\scriptscriptstyle\rm P}}\widehat\xi + \int_0^{\widehat\tau^{\scriptscriptstyle\rm P}} \cK_t\big(\widehat\pi_t-c_t(\nu_t)\big)dt\right]  \\
     &= \dbE^\dbP\left[\cK_{\widehat\tau^{\scriptscriptstyle\rm P}}\big\{\ell_{\widehat\tau}-(U_{\rm P}')^{-1}\big(\widehat\lambda\eta_{\widehat\tau}\big)\big\} + \int_0^{\widehat\tau^{\scriptscriptstyle\rm P}} \cK_t\big\{-(U_{\rm P}')^{-1}\big(\widehat\lambda\eta_t\big)-c_t(\nu_t)\big\}dt\right]  \\
     &= R + \frac{1}{\widehat\lambda}\dbE^\dbP\bigg[\widehat\lambda\Big( \cK_{\widehat\tau^{\scriptscriptstyle\rm P}}\ell_{\widehat\tau^{\scriptscriptstyle\rm P}} 
     - \int_0^{\widehat\tau^{\scriptscriptstyle\rm P}} \!\!\cK_tc_t(\nu_t)dt - R \Big) 
                           - \cK_{\widehat\tau^{\scriptscriptstyle\rm P}}^{\rm P}\widehat\lambda\eta_{\widehat\tau^{\scriptscriptstyle\rm P}}(U_{\rm P}')^{-1}(\widehat\lambda\eta_{\widehat\tau^{\scriptscriptstyle\rm P}}) 
     \\ & \hspace{77mm}       
     - \int_0^{\widehat\tau^{\scriptscriptstyle\rm P}}\!\cK_t^{\rm P}\widehat\lambda\eta_t (U_{\rm P}')^{-1} \!\big(\widehat\lambda\eta_t\big)dt\bigg]  
     \\
     &= R + \frac{1}{\widehat\lambda}\dbE^\dbP\left[\widehat\lambda\mathfrak{H}_{\widehat\tau^{\scriptscriptstyle\rm P}}(\nu) + \mathfrak{J}^{U_{\rm P}^*-U_{\rm P}\circ(U_{\rm P}')^{-1}}_{\widehat\tau^{\scriptscriptstyle\rm P}}\big(\widehat\lambda\big)\right]
      \leq R+\frac{\widehat v}{\widehat\lambda},
 \end{align*}
%
where we used the fact that $U_{\rm P}^*(y)=U_{\rm P}\circ(U_{\rm P}')^{-1}(y)-y(U_{\rm P}')^{-1}(y)$. 
By the arbitrariness of $\dbM\in\cM(\mathbf{\widehat C})$ this provides the upper bound $V^{\rm E}(\mathbf{\widehat C})\le R+\frac{\widehat v}{\widehat\lambda}$, which is achieved by the maximizer of $\widehat v$. 
Hence, a necessary and sufficient condition for equality between the first and the second best contracting problems is that the optimal first best effort $\widehat\dbM$ is also a maximizer of $\widehat v$.

\vspace{4mm}

\noindent (2) \, In the present setting, notice that $\mathfrak{J}^{U_{\rm P}^*}_T(\lambda)$ and $\mathfrak{J}^{U_{\rm P}\circ(U_{\rm P}')^{-1}}_T(\lambda)$ are deterministic. Then, 
  \begin{align*}
  	v_{\rm fb}(\lambda)
  	=
  	\mathfrak{J}^{U_{\rm P}^*}_T(\lambda)
  	+\lambda \sup_{\dbM\in\cM(T,\pi^{\lambda},\xi^{\lambda})} 
  	\dbE^{\dbP}\big[\mathfrak{H}_T(\nu)\big].
  \end{align*}
Similarly, we have
  \begin{align*}
  	\widehat v
  	=
  	\mathfrak{J}^{U_{\rm P}^*-U_{\rm P}\circ(U_{\rm P}')^{-1}}_T(\widehat\lambda)
  	+ \widehat\lambda
  	\sup_{\dbM\in\cM(\mathbf{\widehat C})}
  	\dbE^{\dbP}\big[\mathfrak{H}_{T}(\nu) \big],
  \end{align*}
 which reduces to the same maximization problem as in $v_{\rm fb}(\widehat\lambda)$. Let us finally check that condition (C2) is verified. 
Indeed, notice that in the present case, the optimal controls $(\widehat\tau,\widehat\dbM)=(\tau^\lambda,\dbM^\lambda)$ are independent of $\lambda$. 
Condition (C2) reduces to 
 \begin{align*}
 	0
 	=
 	\dbE^{\widehat\dbP}
 	\Big[\mathfrak{J}^{U_{\rm P}^*}_T(\widehat\lambda)
 	-\mathfrak{J}^{U_{\rm P}\circ(U_{\rm P}')^{-1}}_T(\widehat\lambda)
 	+ \widehat\lambda \mathfrak{H}_T(\widehat\nu)
 	\Big]
 	=
 	V^{\rm PE}_{\rm fb}
 	-\mathfrak{J}^{U_{\rm P}\circ(U_{\rm P}')^{-1}}_T(\widehat\lambda). 
 \end{align*}

Hence, the existence of a unique solution to the last equation follows from our condition \eqref{UP} on the principal's utility function.
\end{proof}

We conclude this section with an example of European contracting problem with risk neutral agent and one single possible stopping $\mathfrak{T}=\{\tau^{\rm P}\}$, where the first best and second best coincide for deterministic $\tau^{\rm P}$, but do not coincide, in general, in the context of a random horizon $\tau^{\rm P}$.

\begin{example}\label{ex:fb=sb}
The Holmstr\"om and Milgrom contracting problem models the output process under effort $\alpha$ by the dynamics 
  $$ dX_t=(rX_t+\alpha_t)dt+dW^\alpha_t, \quad \dbP^\alpha\mbox{-a.s.} $$
We consider the following random horizon extension of the criteria for the agent and the principal, respectively:
 \begin{align*}
 	J(\xi,\alpha)
 	:=\dbE^{\dbP^\alpha}\bigg[\xi e^{-r\tau^{\rm P}}-\frac12\int_0^{\tau^{\rm P}} e^{-rt}\alpha_t^2 dt\bigg],
 	\quad \mbox{and} \quad 
 	J_{\rm P}(\xi,\alpha)
 	:=
 	\dbE^{\dbP^\alpha}\left[e^{-r\tau^{\rm P}} U_{\rm P}(X_T-\xi)\right].
 \end{align*}
Following the same argument as in the previous proof, we find the first best optimal contract is $\widehat\xi:=X_T-(U_{\rm P})^{-1}(\widehat\lambda)$, with corresponding constant optimal effort $\widehat\alpha_t=1$ for all $t\le\tau^{\rm P}$, where the Lagrange multiplier $\hat\lambda$ is the solution of:
 \begin{align*}
   R &= \dbE^{\dbP^{\widehat\alpha}}\Big[e^{-r\tau^{\rm P}}\Big(X_T-(U_{\rm P}')^{-1}\big(\widehat\lambda\big)\Big) -\frac12\big(1-e^{-r\tau^{\rm P}}\big)
   \Big] \\
   & =\frac12\Big(1-\dbE^{\dbP^{\widehat\alpha}}\big[e^{-r\tau^{\rm P}}\big]\Big)
   -\dbE^{\dbP^{\widehat\alpha}}\big[e^{-r\tau^{\rm P}}\big](U_{\rm P}')^{-1}(\widehat\lambda).
 \end{align*}
In order to check whether the first and second best contracts coincide, we need only verify whether the agent's optimal response to the first best contract $\widehat\xi$ is also the unit constant $\widehat\alpha$. Direct calculation provides 
 $$
 J(\widehat\xi,\alpha)
 =
 \dbE^{\dbP^\alpha}\bigg[-e^{-r\tau^{\rm P}}(U_{\rm P}')^{-1}(\widehat\lambda)
                                       +\int_0^{\tau^{\rm P}} e^{-rt}\Big(\alpha_t-\frac12\alpha_t^2\Big) dt\bigg].
 $$ 
For a deterministic finite horizon $\tau^{\rm P}=T$, the maximum of $J(\widehat\xi,\alpha)$ is achieved by the constant unit effort process $\widehat\alpha$, thus proving the identity between the first and the second best problem. 

However, this is not the case anymore for random horizon $\tau^{\rm P}$, in general. Consider, for instance, the example $\tau^{\rm P}:=\inf\{t>0: X_t\le 0\}$, with $X_0>0$. By standard control theory, the HJB equation corresponding to this problem is 
  $$ rv-\frac12 v''-\sup_a\left\{a(1+v')-\frac12a^2\right\}=0, \quad \mbox{on~}\dbR_+, $$ 
  with boundary condition $v(0)=-(U_{\rm P}')^{-1}(\widehat\lambda)$. Suppose to the contrary that the unit constant effort $\alpha$ is optimal, then $\hat a=1$ must be the maximizer in the last HJB equation which happens if and only if $v'=0$, meaning that the value function $v$ is constant, but the HJB equation then reduces to $v=\frac1{2r}$ which does not match the boundary condition at the origin.
\end{example}

\section{Reduction to a standard stochastic control problem}
\label{sec:mainresult}

In this section, we extend the result of Cvitani\'c, Possama\"i, and Touzi \cite{CPT18} to the present random horizon setting. The key argument, introduced by Sannikov \cite{Sannikov08},
is to reduce the principal optimization problem by using the dynamic programming representation of the agent's value process. As is standard in stochastic control theory, such a representation involves the agent's (path-dependent) Hamiltonian:
   \begin{equation} \label{eq:defH}
     H_t(\omega,y,z,\gamma) := \underset{u\in A\times B}{\sup} h_t(\omega,y,z,\gamma,u);~~(t,\omega)\in[0,\infty)\times\Omega,~(y,z,\gamma)\in\dbR\times\dbR^d\times\cS_d(\dbR),
   \end{equation}
   where $\mathcal S_d(\dbR)$ is the set of symmetric matrices in $\cM_{d}(\dbR)$, and for $u=(a,b)\in A\times B$
   \begin{equation} \label{eq:defh}
     h_t(\omega,y,z,\gamma,u):= -c_t(\omega,u)-k_t(\omega,u)y + \sigma_t(\omega,b)\lambda_t(\omega,a)\cdot z + \frac{1}{2}{\rm Tr}\big[(\sigma_t\sigma_t^\top)(\omega,b)\gamma\big],
   \end{equation}
where ${\rm Tr}[M]$ denotes the trace of a matrix $M\in\mathcal M_d(\dbR)$.

We next introduce for an arbitrary initial value $Y_0\in\dbR$ and $\dbF$-predictable processes $(Z,\Gamma)$ with values in $\dbR^d\times\cS_d(\dbR)$ the process $Y^{Y_0,Z,\Gamma}$ defined by the random ODE:
 \begin{equation} \label{def:Y}
           Y^{Y_0,Z,\Gamma}_t 
           :=
           Y_0 + \int_0^t \left(Z_r\cdot dX_r+\frac12{\rm Tr}\big[\Gamma_rd\langle X\rangle_r\big]
                                 -H_r\big(Y^{Y_0,Z,\Gamma}_r,Z_r,\Gamma_r\big)dr-U_{\rm A}(\pi_r)dr\right)
         \end{equation}
  under appropriate integrability. We shall see that the process $Y^{Y_0,Z,\Gamma}$ turns out to represent the agent's value process, 
  and will be shown to be a convenient parameterization of the contracts by setting $(\tau^{\scriptscriptstyle\rm P},\pi,\xi)=(\tau^{\scriptscriptstyle\rm P},\pi,\xi^{Y_0,Z,\Gamma})$ with $\xi^{Y_0,Z,\Gamma}:=U_{\rm A}^{-1}\big(Y^{Y_0,Z,\Gamma}_{\tau^{\scriptscriptstyle\rm P}}\big)$.

\begin{definition} \label{def:V}
We denote by $\cV$ the collection of all such processes $(Z,\Gamma)$ satisfying, in addition, the following:
 \begin{enumerate}[{\rm (i)}]
  \item $\|Z\|_{\cH^p_{\alpha,\tau^{\scriptscriptstyle\rm P}}(\cP)} 
            + \|Y^{Y_0,Z,\Gamma}\|_{\cD^p_{\alpha,\tau^{\scriptscriptstyle\rm P}}(\cP)}<\infty$, for some $p>1$ and $\alpha\in\dbR$.
  \item There exists a weak solution $\big(\dbP^{Y_0,Z,\Gamma},\nu^{Y_0,Z,\Gamma}\big)\in\cM$ such that
         \begin{equation} \label{subgrad}
           H_t(Y_t,Z_t,\Gamma_t) = h_t\big(Y_t,Z_t,\Gamma_t,\nu^{Y_0,Z,\Gamma}_t\big),
           ~~dt\otimes\dbP^{Y_0,Z,\Gamma}\mbox{-a.e. on}~~
           \llbracket 0,\t^{\scriptscriptstyle\rm P}\rrbracket.
         \end{equation} 
 \end{enumerate}
\end{definition}

Condition (i) guarantees that the process $Y^{Y_0,Z,\Gamma}$ of \eqref{def:Y} is well-defined $\dbP$-a.s.~for all $\dbP\in\cP$. 
First, as $k$ is bounded, the Hamiltonian $H$ is Lipschitz in the $y$ variable. 
It guarantees that $Y^{Y_0,Z,\Gamma}$ is well-defined as the unique solution of the ODE with random coefficients \eqref{def:Y}, provided that the integrals are well-defined. 
Moreover, as in \cite{CPT18}, the integrals are indeed well-defined, without further condition on the process $\Gamma$, as we see by applying It\^o's formula that 
  \begin{equation} \label{YZK}
   \cK^{\nu}_tY^{Y_0,Z,\Gamma}_t + \int_0^t \cK^{\nu}_r\big(U_{\rm A}(\pi_r) - c_r(\nu_r)\big)dr = Y_0 + \int_0^t \cK^{\nu}_rZ_r\cdot\sigma_r(\beta_r)dW^\dbP_r-A^\nu_t, ~t\leq\tau,~\dbP\mbox{-a.s.}, 
  \end{equation}
  for all $(\dbP,\nu)\in\cM$, where $A^\nu:= \int_0^. \cK^{\nu}_r\big[H_r-h_r(.,\nu_r)\big]\big(Y_r^{Y_0,Z,\Gamma},Z_r,\Gamma_r\big)dr$ is a nondecreasing process. 
  Due to Assumption \ref{assum:costc} and the admissibility condition \eqref{contract-growth}, the first integral is well-defined. 
  Now, the only issue is with the existence of the stochastic integral $\int_0^.\cK_r^\nu Z_r\cdot\sigma_r(\beta_r)dW^\dbP_r$ under each $\dbP\in\cP$. 
We emphasize that, as a consequence of the main result of Nutz \cite{Nut12}, the stochastic integral $\int_0^.\cK_r^\nu Z_r\cdot dX_r$ is defined pathwisely on $\Omega$ without exclusion of any null set. 
This is a crucial fact as our main result below states that the principal's problem can be reduced to choosing among contracts of the form $\big(\tau^{\scriptscriptstyle\rm P},\pi,U_{\rm A}^{-1}(Y^{Y_0,Z,\Gamma}_{\tau^{\scriptscriptstyle\rm P}})\big)$, 
   which requires that such contracts be independent from the agent's control model.

Condition (ii) states the existence of a maximizer of the Hamiltonian $H$, defined in \eqref{eq:defH}, that induces an admissible control model for the agent's problem. 
The existence of a maximizer is a standard condition in the verification argument in stochastic control theory, which allows one to identify the optimal control. 
As in \cite{CPT18}, we shall see that, given $\Cbf=\big(\tau^{\scriptscriptstyle\rm P},\pi,U_{\rm A}^{-1}(Y^{Y_0,Z,\Gamma}_{\tau^{\scriptscriptstyle\rm P}})\big)$, the process $Y^{Y_0,Z,\Gamma}$ is the dynamic value function of the agent's control problem, 
   and is precisely expressed in the required It\^o decomposition form \eqref{def:Y}. 
In particular, $Y_0=V^{\rm E}(\Cbf)$. 
As the principal problem restricts to those admissible contracts which induce existence for the agent's problem $\widehat{\cM}^{\rm E}(\Cbf)\neq\emptyset$, 
   condition (ii) is necessary to characterize the agent's optimal response which needs to be plugged in the principal's problem $V^{\rm PE}$. 
A similar discussion applies to the American principal-agent problem.
 
By Condition (ii) together with the continuity of $h$, we deduce from a classical measurable selection argument (see e.g.~\cite{Benes70, Benes71}), 
   the existence of measurable maps $\widehat u_t(\omega,y,z,\gamma):=(\widehat\alpha,\widehat\beta)_t(\omega,y,z,\gamma)$ which maximize $H$
     $$ H_t(\omega,y,z,\gamma) = h_t\big(\omega,y,z,\gamma,\widehat u_t(\omega,y,z,\gamma)\big). $$
We next denote by $\widehat{\mathcal U}$ the collection of all such measurable maximizers, and we introduce the optimal feedback controls
     $$ \widehat{\nu}^{Y_0,Z,\Gamma}_t:=\widehat u_t\big(X,Y^{Y_0,Z,\Gamma}_t,Z_t,\Gamma_t\big), 
     $$ 
which induce the following coefficients for the optimal output process     
   \begin{equation*}
     \widehat{\lambda}_t(\omega,y,z,\gamma) := \lambda_t\big(\omega,\widehat{\alpha}_t(\omega,y,z,\gamma)\big), \quad 
     \widehat{\sigma}_t(\omega,y,z,\gamma) := \sigma_t\big(\omega, \widehat{\beta}_t(\omega,y,z,\gamma)\big).
   \end{equation*}
By Condition (ii) of Definition \ref{def:V}, it follows that for all $(Z,\Gamma)\in\cV$ and any $\widehat u\in\widehat{\mathcal U}$, the following stochastic differential equation driven by a $d$-dimensional Brownian motion $W$
  \begin{equation} \label{eq:OptimalSDE}
    X_t = X_0 +\int_0^t \widehat{\sigma}_r(X,Y_r^{Y_0,Z,\Gamma},Z_r,\Gamma_r)\big(\widehat{\lambda}_r(X,Y_r^{Y_0,Z,\Gamma},Z_r,\Gamma_r)dr + dW_r\big),~~t\leq\tau,
  \end{equation}
  has at least one weak solution $\widehat{\dbM}^{Y_0,Z,\Gamma}=(\widehat{\dbP}^{Y_0,Z,\Gamma},\widehat{\nu}^{Y_0,Z,\Gamma})$. 
Our main result is the following extension of Cvitani\'c, Possama\"{\i}, and Touzi \cite{CPT18} reduction result to the present random horizon context. 
Recall the notation $\xi^{Y_0,Z,\Gamma}:=U_{\rm A}^{-1}(Y^{Y_0,Z,\Gamma}_{\tau^{\scriptscriptstyle\rm P}})$ for $(Y_0,Z,\Gamma)$ ranging in $\dbR\times\cV$.

\begin{theorem} \label{thm:main}
Assume that $\cV\neq\emptyset$. Then, 
\begin{enumerate}[{\rm (i)}]
 \item $V^{\rm PE} = \sup_{Y_0\ge R} \underline{V}^{\rm PE}(Y_0)$, where
   \begin{equation*}
     \underline{V}^{\rm PE}(Y_0) 
      := 
     \sup_{\substack{(\tau^{\scriptscriptstyle\rm P},\pi)\in\cT\times\Pi  \\  (Z,\Gamma)\in\cV}}
     \sup_{(\dbP,\nu)\in\widehat{\cM}^{\rm E}\big(\tau^{\scriptscriptstyle\rm P},\pi,\xi^{Y_0,Z,\Gamma}\big)} \!\!\!
     \dbE^{\dbP}\bigg[\cK^{\rm P}_{\tau^{\scriptscriptstyle\rm P}} U_{\rm P}\big(\ell_{\tau^{\scriptscriptstyle\rm P}}-\xi^{Y_0,Z,\Gamma}\big)
                                 +\int_0^{\tau^{\scriptscriptstyle\rm P}}\!\!\!\cK_r^{\rm P}U_{\rm P}(-\pi_r)dr\bigg].
   \end{equation*}
   Moreover, if $(Y^*_0,Z^*,\Gamma^*,\tau^*,\pi^*)$ is a solution of the last optimal control problem, then the triple $(\tau^*,\pi^*,\xi^{Y^*_0,Z^*,\Gamma^*})$ is an optimal contract for the European principal-agent problem.
 \item $V^{\rm PA} = \sup_{Y_0\ge R} \underline{V}^{\rm PA}(Y_0)$,
        where, denoting $\mbox{\rm\sc h}_0:=\mbox{\rm\sc h}_0^{Y_0,Z,\Gamma}:= \inf \big\{t \ge 0:\, Y^{Y_0,Z,\Gamma}_t \le U(\rho) \big\}$,
          \begin{equation*}
             \underline{V}^{\rm PA}(Y_0) 
               := 
             \sup_{\substack{\pi \in \Pi \\ (Z,\Gamma)\in\cV}} 
             \sup_{(\dbP,\nu)\in\widehat{\cM}^{\rm A}(\mbox{\rm\sc h}_0,\pi)}
                       \dbE^{\dbP}\bigg[\cK^{\rm P}_{\mbox{\rm\sc h}_0} U_{\rm P}\big(\ell_{\mbox{\rm\sc h}_0} \big)
                                 +\int_0^{\mbox{\rm\sc h}_0} \cK_r^{\rm P}U_{\rm P}(-\pi_r)dr\bigg].
          \end{equation*}
   Moreover, if $(Y^*_0,Z^*,\Gamma^*,\pi^*)$ is a solution of the last optimal control problem, then denoting $\tau^*:=\mbox{\rm\sc h}_0^{Y^*_0,Z^*,\Gamma^*}$, 
       the pair $(\tau^*,\pi^*)$ is an optimal contract for the American principal-agent problem.
\end{enumerate}

\begin{remark}
Once proving the main theorem above, we can treat the principal's problem as standard stochastic control problem, using dynamic programming arguments. If the coefficients of the principal's problem are Markovian, then the dynamic programming principle links the control problem to the HJB (see e.g.~the examples in Section \ref{sec:example}). Otherwise, in case that the coefficients are path-dependent, other tools such as backward stochastic differential equation (BSDE) \cite{EKPQ97}, second-order BSDE \cite{PTZ18}, backward stochastic PDE \cite{MT2010} and path-dependent PDE \cite{RTZ2014} can be used to characterize the value functions. 
\end{remark}

\end{theorem}

The key argument for this reduction result is the following density property of the class of contracts $\Cbf=(\tau,\pi,\xi^{Y_0,Z,\Gamma})$.

\begin{proposition} \label{prop:density}
 Let $\Cbf=(\tau,\pi,\xi)\in\mathfrak{C}^{\rm E}_R$. Then, we may find $Y_0^\varepsilon\geq R$ and $(Z^\varepsilon,\Gamma^\varepsilon)\in\cV$ such that, with $\xi^\varepsilon:=U_{\rm A}^{-1}\big(Y^{Y^\varepsilon_0,Z^\varepsilon,\Gamma^\varepsilon}_\tau\big)$, we have 
 $$
 \Cbf^\varepsilon:=(\tau,\pi,\xi^\varepsilon)\in\mathfrak{C}^{\rm E}_R,
 ~~\widehat{\cM}^{\rm E}(\Cbf^\varepsilon)=\widehat{\cM}^{\rm E}(\Cbf),
 ~~\mbox{and}
 ~~\xi^\varepsilon=\xi,~\dbP\mbox{-a.s., for all}~(\dbP,\nu)\in\widehat{\cM}^{\rm E}(\Cbf).
 $$
\end{proposition} 
  
We postpone the proof of this result to the next section, and we use it now for the proof of Theorem \ref{thm:main} {\rm (i)} and (ii). 

\begin{proof}[Proof of Theorem \ref{thm:main} {\rm (i)}]
We organize the proof in two steps. 
We first establish inequality $V^{\rm PE}\geq\underline{V}^{\rm PE}(Y_0)$ by following the classical verification argument in stochastic control theory, and we next prove equality by using the density result of Proposition \ref{prop:density}.
 
 \vspace{3mm}
 
 \noindent \textbf{Step 1.} We first show that $V^{\rm PE}\geq\underline{V}^{\rm PE}(Y_0)$, for all $Y_0\in\dbR$. 
    Let $(Z,\Gamma)\in\cV$, and fix some stopping time $\tau^{\scriptscriptstyle\rm P}$, and optional process $\pi$ satisfying the integrability condition in \eqref{contract-growth}. 
    The required inequality is a direct consequence of the following two steps. 
    
 \vspace{2mm}  
   
\noindent \textbf{1.a.} We first verify that $\Cbf^{Y_0,Z,\Gamma}=\big(\tau^{\scriptscriptstyle\rm P},\pi,\xi^{Y_0,Z,\Gamma}\big)\in\mathfrak{C}$, $(\dbP^{Y_0,Z,\Gamma},\nu^{Y_0,Z,\Gamma})\in\widehat\cM^{\rm E}\big(\Cbf^{Y_0,Z,\Gamma}\big)$ 
                                    and $Y_0=V^{\rm E}\big(\Cbf^{Y_0,Z,\Gamma}\big)$.
    From the definition of $Y_{\tau^{\scriptscriptstyle\rm P}}^{Y_0,Z,\Gamma}$ in \eqref{def:Y}, it is clear that $\xi^{Y_0,Z,\Gamma}$ is an $\cF_{\tau^{\scriptscriptstyle\rm P}}$-measurable random variable. 
    The integrability of $Y_{\tau^{\scriptscriptstyle\rm P}}^{Y_0,Z,\Gamma}=U_{\rm A}(\xi^{Y_0,Z,\Gamma})$ follows from Definition \ref{def:V} (i). 

    For any $\dbM=(\dbP,\nu)\in\cM$, it follows from a direct application of It\^o's formula that 
     \begin{align*}
       \cK^\nu_{\tau^{\scriptscriptstyle\rm P}} Y^{Y_0,Z,\Gamma}_{\tau^{\scriptscriptstyle\rm P}} = Y_0 & + \int_0^{\tau^{\scriptscriptstyle\rm P}} \cK^\nu_r Z_r\cdot\sigma_r^{\beta_r}dW_r^{\dbP} - \int_0^{\tau^{\scriptscriptstyle\rm P}} \cK^\nu_rH_r\big(Y_r^{Y_0,Z,\Gamma},Z_r,\Gamma_r\big)dr - \int_0^{\tau^{\scriptscriptstyle\rm P}} \cK^\nu_rU_{\rm A}(\pi_r)dr \\
                                            & + \int_0^{\tau^{\scriptscriptstyle\rm P}} \cK^\nu_r\Big(-k_r^{\nu_r}Y_r^{Y_0,Z,\Gamma} + Z_r\cdot\sigma^{\beta_r}_r\lambda_r^{\alpha_r} + \frac{1}{2}{\rm Tr}\big[\widehat \sigma_r^2\Gamma_r\big]\Big)dr,
     \end{align*}
     where we used the simplifying notation $\varphi_r^u:=\varphi_r(x,u)$ for $\varphi=k,$ $\sigma,$ $\lambda$.
    As $(Z,\Gamma)\in\mathcal V_0$, the stochastic integral $\int_0^\cdot \mathcal K^{\nu}_r Z_r\cdot \sigma_r^{\beta_r}dW^\dbP_r$ defines a martingale. 
    By the definition of the agent's optimization criterion $J^{\rm E}$ and the definition of $h$, we may write the last equation as
     \begin{align} \label{ineq-A}
       J^{\rm E}\big(\dbM,\Cbf^{Y_0,Z,\Gamma}\big) &= \dbE^\dbP\left[\cK^\nu_{\tau^{\scriptscriptstyle\rm P}} U_{\rm A}\big(\xi^{Y_0,Z,\Gamma}_{\tau^{\scriptscriptstyle\rm P}}\big) + \int_0^{\tau^{\scriptscriptstyle\rm P}} \cK^\nu_r\big(U_{\rm A}(\pi_r)-c_r(\nu_r)\big)dr\right] \nonumber \\
                                             &= Y_0 - \dbE^{\dbP}\left[\int_0^{\tau^{\scriptscriptstyle\rm P}}\cK_r^{\nu}\Big(H_r\big(Y_r^{Y_0,Z,\Gamma},Z_r,\Gamma_r\big) - h_r\big(Y_r^{Y_0,Z,\Gamma},Z_r,\Gamma_r,\nu_r\big)\Big)dr\right].                                 
     \end{align}
    It follows by the definition of $H$ that $J^{\rm E}\big(\dbM,\Cbf^{Y_0,Z,\Gamma}\big)\leq Y_0$, and thus $V^{\rm E}\big(\Cbf^{Y_0,Z,\Gamma}\big)\leq Y_0$ by the arbitrariness of $\dbM\in\cM$. 
    Finally, the equality $J^{\rm E}\big(\dbP^{Y_0,Z,\Gamma},\nu^{Y_0,Z,\Gamma},\Cbf^{Y_0,Z,\Gamma}\big)=Y_0$ holds in \eqref{ineq-A} with the control $(\dbP^{Y_0,Z,\Gamma},\nu^{Y_0,Z,\Gamma})$ introduced in the admissibility condition (ii) of Definition \ref{def:V}.  
    This shows that $(\dbP^{Y_0,Z,\Gamma},\nu^{Y_0,Z,\Gamma})\in\widehat\cM^{\rm E}\big(\Cbf^{Y_0,Z,\Gamma}\big)\neq\emptyset$, and therefore $\Cbf^{Y_0,Z,\Gamma}\in\frakC$.
 
 \vspace{2mm}
  
\noindent \textbf{1.b.} We next show $(\widehat{\dbP},\widehat{\nu})\in\widehat{\cM}^{\rm E}\big(\Cbf^{Y_0,Z,\Gamma}\big)$ 
      if and only if $H_t(Y_t,Z_t,\Gamma_t)=h_t(Y_t,Z_t,\Gamma_t,\widehat{\nu}_t)$, $dt\otimes\widehat{\dbP}$-a.e.~on $\llbracket 0,\tau^{\scriptscriptstyle\rm P}\rrbracket$, 
      i.e., the control process $\widehat{\nu}$ is a maximizer of the Hamiltonian on the support of $\widehat{\dbP}$. 
   It follows from \eqref{ineq-A} and the equality $V^{\rm E}\big(\Cbf^{Y_0,Z,\Gamma}\big)= Y_0$, established in Step 1.a, that we must have for all $(\widehat\dbP,\widehat\nu)\in\widehat\cM^{\rm E}\big(\Cbf^{Y_0,Z,\Gamma}\big)$ that
     $$ \dbE^{\widehat\dbP}\left[\int_0^{\tau^{\scriptscriptstyle\rm P}} \cK_r^{\widehat\nu} \Big(H_r\big(Y_r^{Y_0,Z,\Gamma},Z_r,\Gamma_r\big) - h_r\big(Y_r^{Y_0,Z,\Gamma},Z_r,\Gamma_r,\widehat\nu_r\big)\Big)dr\right] = 0. $$
   By the definition of $H$ in \eqref{eq:defH}, this holds if and only if $\widehat\nu$ is a maximizer of $H_r\big(Y_r^{Y_0,Z,\Gamma},Z_r,\Gamma_r\big)$, $dt\otimes\widehat\dbP$-a.e.~on $\llbracket 0,\tau^{\scriptscriptstyle\rm P}\rrbracket$.
 
 \vspace{2mm}
 
 \noindent To summarize: for $(\tau^{\scriptscriptstyle\rm P},\pi)\in\cT\times\Pi$, $Y_0\geq R$ and $(Z,\Gamma)\in\cV$, we have that $\Cbf^{Y_0,Z,\Gamma}=(\tau^{\scriptscriptstyle\rm P},\pi,\xi^{Y_0,Z,\Gamma})\in\mathfrak C$, i.e., $\Cbf^{Y_0,Z,\Gamma}$ is an admissible contract, 
            and $\widehat\cM^{\rm E}(\Cbf^{Y_0,Z,\Gamma})\neq\emptyset$ as well as $V^{\rm E}(\Cbf^{Y_0,Z,\Gamma})=Y_0$. 
           Therefore, it follows immediately that $V^{\rm PE}\geq\sup_{Y_0\geq R}\underline{V}^{\rm PE}(Y_0)$.
 
\vspace{3mm}
  
 \noindent \textbf{Step 2.} By Proposition \ref{prop:density}, for any $\Cbf=(\tau^{\scriptscriptstyle\rm P},\pi,\xi)\in\mathfrak{C}_R^{\rm E}$ with $\widehat{\cM}^{\rm E}\neq\emptyset$, we may define a contract $\Cbf^\varepsilon=(\tau^{\scriptscriptstyle\rm P},\pi,\xi^\varepsilon)\in\mathfrak{C}_R^{\rm E}$, 
              where $\xi^\varepsilon = U_{\rm A}^{-1}\big(Y^{Y_0^\varepsilon,Z^\varepsilon,\Gamma^\varepsilon}_{\tau^{\scriptscriptstyle\rm P}}\big)$ for some $(Z^\varepsilon,\Gamma^\varepsilon)\in\cV$, 
              such that $\widehat\cM^{\rm E}(\Cbf^\varepsilon) = \widehat\cM^{\rm E}(\Cbf)$ and $\xi^\varepsilon=\xi$, $\widehat\dbP$-a.s.~for all $(\widehat\dbP,\widehat\nu)\in\widehat{\cM}^{\rm E}(\Cbf)$. 
           Therefore, for each $(\widehat\dbP,\widehat\nu)\in\widehat\cM^{\rm E}(\Cbf)=\widehat\cM^{\rm E}(\Cbf^\varepsilon)$ we obtain that 
            \begin{align*}
              J^{\rm P}(\Cbf^\varepsilon) &= \sup_{(\widehat\dbP,\widehat\nu)\in\widehat\cM^{\rm E}(\Cbf^\varepsilon)}\dbE^{\widehat\dbP}\left[\cK^{\rm P}_{\tau^{\scriptscriptstyle\rm P}} U_{\rm P}(\ell_{\tau^{\scriptscriptstyle\rm P}}-\xi^\varepsilon)+\int_0^{\tau^{\scriptscriptstyle\rm P}}\cK^{\rm P}_r U_{\rm P}(-\pi_r)dr\right] \\
                                          &= \sup_{(\widehat\dbP,\widehat\nu)\in\widehat\cM^{\rm E}(\Cbf)}\dbE^{\widehat\dbP}\left[\cK^{\rm P}_{\tau^{\scriptscriptstyle\rm P}} U_{\rm P}(\ell_{\tau^{\scriptscriptstyle\rm P}}-\xi)+\int_0^{\tau^{\scriptscriptstyle\rm P}}\cK^{\rm P}_r U_{\rm P}(-\pi_r)dr\right] = J^{\rm P}(\Cbf).
            \end{align*}
By Step 1, notice that, the agent's problem with the contract $\Cbf^\varepsilon$ can be explicitly solved and we obtain $V^{\rm A}(\Cbf^\varepsilon)=Y_0^\varepsilon$.            
           By arbitrariness of $\Cbf$, we obtain that $V^{\rm PE}\leq\sup_{Y_0\geq R}\underline{V}^{\rm PE}(Y_0)$.
\end{proof}

In order to obtain a similar reduction result for the American principal-agent problem, we follow Sannikov's \cite{Sannikov08} idea by proceeding to a first reduction of the principal problem which allows to transform the corresponding agent problem into that of a European contract as no early exercise is optimal for the agent. 
 
\begin{proof}[Proof of Theorem \ref{thm:main} {\rm (ii)}] 
 Similar to the proof of Theorem \ref{thm:main} (i), we proceed in three steps, following the classical verification argument in stochastic control theory.

\vspace{3mm}

\no {\bf Step 1.}\q 
  We first prove that $ V^{\rm PA} \ge \sup_{Y_0\ge R} \underline{V}^{\rm PA}(Y_0)$. 
  Let $Y_0\ge R$, $(Z,\Gamma)\in\cV$, $\pi\in \Pi$, and $\mbox{\sc h}_0:=\mbox{\sc h}_0^{Y_0,Z,\Gamma}=\inf\{t\ge 0: Y^{Y_0,Z,\Gamma}\le U_{\rm A}(\rho)\}\in[0,\infty]$ be as defined in the statement of the theorem, 
     and consider the principal contract $\Cbf:=(\mbox{\sc h}_0, \pi,\rho)$. 
  For $\dbM\in \cM$ and $\t \le \mbox{\sc h}_0$ we have
     \begin{align*}
       J^{\rm A}\big(\t, \dbM,\Cbf\big) 
        &= \dbE^\dbP\left[\cK^\nu_\tau U_{\rm A}(\rho)  + \int_0^\tau \cK^\nu_r\big(U_{\rm A}(\pi_r)-c_r(\nu_r)\big)dr\right] \\
        & \le \dbE^\dbP\left[\cK^\nu_\tau Y^{Y_0,Z,\Gamma}_\t + \int_0^\tau \cK^\nu_r\big(U_{\rm A}(\pi_r)-c_r(\nu_r)\big)dr\right] \\
        &= Y_0 - \dbE^{\dbP}\left[\int_0^\tau\cK_r^{\nu} \big(H_r- h_r(.,\nu_r)\big)\big(Y_r^{Y_0,Z,\Gamma},Z_r,\Gamma_r\big) dr\right] \le Y_0.
     \end{align*}
  The last inequality is due to the definition of $H$. Moreover, as $\tau\le\mbox{\sc h}_0$, it is clear that the only way to turn both inequalities above into equalities is to take 
    \begin{align} \label{verif-American}
       \dbM^{Y_0,Z,\Gamma} = \big(\dbP^{Y_0,Z,\Gamma},\nu^{Y_0,Z,\Gamma}\big) \quad \mbox{and} \quad
       \widehat\t =\mbox{\sc h}_0,~\dbP^{Y_0,Z,\Gamma}\mbox{-a.s.}, 
    \end{align}
    where we use the notations of Definition \ref{def:V}, together with the condition that the set $\cV$ is non-empty. 
  Therefore $V^{\rm A}(\Cbf)=Y_0$ with optimal American agent response given by the pair $\big(\mbox{\sc h}_0, \dbM^{Y_0,Z,\Gamma}\big)$. 
  By the same argument as in {\bf Step 1} of the proof of Theorem \ref{thm:main} (i), this provides the inequality $V^{\rm PA} \ge \sup_{Y_0\ge R} \underline{V}^{\rm PA}(Y_0)$.

\vspace{3mm}

\no {\bf Step 2.}\q In order to prove that equality holds, we introduce the dynamic version of the American agent problem for an arbitrary $\Cbf^{\rm A}=(\tau^{\scriptscriptstyle\rm P},\pi)$: 
 \beaa
   V^{\rm A}_t(\Cbf^{\rm A}) := \esssup_{\t\ge t,\hspace{0.5mm}\dbM\in\cM} 
                          \dbE^{\dbP}_t\left[\cK^{\nu}_{t,\tau\we\t^{\scriptscriptstyle\rm P}}U_{\rm A}(\rho)
                                +\int_t^{\tau\we\t^{\scriptscriptstyle\rm p}} \mathcal K^{\nu}_{t,s} \big(U_{\rm A}(\pi_s)-c_s(\nu_s)\big)ds\right],
 \eeaa
  where $\cK^\nu_{t,s}:=(\cK^\nu_t)^{-1}\cK^\nu_s$. 
 Then define
   \bea\label{eq:optimalstopping}
     \widehat\t:=\inf\Big\{t\ge 0:~\overline{V}^{\rm A}_t \le U_{\rm A}(\rho)\Big\},\q\mbox{where}\,\,\, \overline{V}^{\rm A}_t = \lim_{s\downarrow t, s\in \dbQ}V^{\rm A}_s.
   \eea
Note that $\widehat\t\le\tau^{\scriptscriptstyle\rm P}$. We claim and shall prove in {\bf Step 3} that $\widehat\t$ is an optimal stopping time for the agent, i.e.
 \bea\label{eq:VAdefn}
  V^{\rm A}_0(\Cbf^{\rm A}) 
     =  
   \sup_{\dbM\in\cM} 
   \dbE^{\dbP}\bigg[\cK^{\nu}_{\widehat\t} 
                                U_{\rm A}(\rho)
                                +\int_0^{\widehat\t} \mathcal K^{\nu}_{s} \big(U_{\rm A}(\pi_s)-c_s(\nu_s)\big)ds 
                       \bigg].
 \eea
Therefore, we may reduce the principal to offer contracts of the form $\mathbf{C}^{\rm A}=(\widehat\t,\pi)$, as her utility criterion is not changed by fixing $\t^{\scriptscriptstyle\rm P}:=\widehat\t$, and the agent's problem reduces to 
   $$ V^{\rm A}\big(\Cbf^{\rm A}\big) = \sup_{ \dbM\in\cM} J^{\rm A}\big(\widehat\t, \dbM,\Cbf^{\rm A}\big) 
                                 = \sup_{\dbM\in\cM} J^{\rm E}\big(\dbM,\Cbf\big),\quad \mbox{with}~~\Cbf:=\big(\Cbf^{\rm A},\rho\big). $$  

We have thus transformed the American agent problem into a stochastic control problem (without optimal stopping) as in the European agent context of Theorem \ref{thm:main} (i), and we may now continue by adapting the same argument as in {\bf Step 2} of the proof of Theorem \ref{thm:main} (i). 
Namely, Proposition \ref{prop:density} guarantees the existence of a contract $\Cbf^\varepsilon=(\widehat\tau,\pi,\xi^\varepsilon)\in\mathfrak{C}_R$, 
   where $\xi^\varepsilon = U_{\rm A}^{-1}\big(Y^{Y_0^\varepsilon,Z^\varepsilon,\Gamma^\varepsilon}_\tau\big)$ for some $(Z^\varepsilon,\Gamma^\varepsilon)\in\cV$, 
   such that $\widehat\cM^{\rm E}(\Cbf^\varepsilon) = \widehat\cM^{\rm E}(\Cbf)$ and $\xi^\varepsilon=\rho$, $\widehat\dbP$-a.s.~for all $(\widehat\dbP,\widehat\nu)\in\widehat{\cM}^{\rm E}(\Cbf)$. 
Next, define the new contract $\Cbf^\e :=(\widehat{\t}^\e, \pi,\rho)$ where $\widehat{\t}^\e:=\widehat{\t}\wedge\inf\{t\ge 0:~ Y^\e_t \leq U_{\rm A}(\rho)\}$, and we observe that for all $(\widehat\dbP,\widehat\nu)\in \widehat \cM^{\rm E}(\Cbf) =\widehat \cM^{\rm E}(\Cbf^\e)$, we have $\widehat{\t}^\e = \widehat\t$, $\widehat\dbP$-a.s., which is exactly the condition \eqref{verif-American} required for the verification argument in Step 1 of the present proof.

We continue the proof by following exactly the same line of argument as in {\bf Step 2} of the proof Theorem \ref{thm:main} (i), and we obtain the required equality.

\vspace{3mm}

\no {\bf Step 3.}\q  Here we are going to complete the proof by showing that $\widehat\t$ in \eqref{eq:optimalstopping} is the optimal stopping time for the agent. First, by the definition of $V^{\rm A}$, we have for any $t'\ge t$
  \begin{equation*}
    V^{\rm A}_t(\Cbf^{\rm A}) 
     \ge
     \dbE^{\dbP}_t\left[\cK^{\nu}_{t,t'}V^{\rm A}_{t'}\big(\Cbf^{\rm A}\big)\right].
  \end{equation*}
 Therefore, $\cK^\nu_{0,t}V^{\rm A}_t$ is a $\dbP$-supermartingale for all $(\dbP,\nu)\in \cM$. 
 Then, it is a classical result (see e.g.~\cite[Proposition 1.3.14]{KS91}) that the right limit of the process $V^{\rm A}$ exists $\dbP$-a.s.~for all $\dbP\in \cP$. 
 In particular, the process $\overline{V}^{\rm A}$ defined in \eqref{eq:optimalstopping} is right-continuous $\dbP$-a.s.~for all $\dbP\in \cP$, and thus $\widehat\t$ is a stopping time. 
 Further, let $\big(\widehat \dbP,\widehat \nu\big)$ be an optimal control, and thus
  \begin{equation*}
    V^{\rm A}_t(\Cbf^{\rm A}) 
      =
      \esssup_{\t\ge t} 
      \dbE^{\widehat\dbP}_t\bigg[\cK^{\widehat\nu}_{t,\tau\we\t^{\scriptscriptstyle\rm P}} 
                                U_{\rm A}(\rho)
                                +\int_t^{\tau\we\t^{\scriptscriptstyle\rm p}} \mathcal K^{\widehat\nu}_{t,s} \big(U_{\rm A}(\pi_s)-c_s(\widehat\nu_s)\big)ds 
                           \bigg].
  \end{equation*}
 It follows the standard result of optimal stopping that the optimal stopping time is equal to $\widehat\t$, $\widehat\dbP$-a.s. 
 Therefore, we obtain \eqref{eq:VAdefn}.
\end{proof}

\section{Examples} \label{sec:example}

\subsection{Sannikov \cite{Sannikov08}}

This section reports our understanding of the model in Sannikov \cite{Sannikov08}. 
Given a European contract $\Cbf=(\tau,\pi,\xi)$ proposed by the principal, the agent has a nonnegative increasing strictly concave utility function $U_{\rm A}$ and a nonnegative increasing convex cost function $h$, and is solving:
 \beaa
 \sup_{\alpha} \dbE^{\dbP^{\alpha}}\left[e^{-r\tau}U_{\rm A}(\xi)+\int_0^\tau e^{-rt}\big(U_{\rm A}(\pi_t)-h(\alpha_t)\big) dt\right], 
 \eeaa
 where
  $$ X_t=X_0+\int_0^t \alpha_sds+dW^\alpha_s, ~t\geq 0,\quad \dbP^\alpha\mbox{-a.s.}, $$
and, as in the previous example, the agent's effort $\alpha$ is an arbitrary progressively measurable process taking values in some subset $A\subseteq\dbR$ and satisfying $\dbE^{\dbP^0}\big[\mathrm{D}^{\dbP^\alpha |\dbP^0}_T\big]=1$. 

The Hamiltonian is given by
 \beaa
 H(y,z,\gamma)
 =
 -ry+
 \frac12{\rm Tr}[\gamma]+H^0(z),
 &\mbox{where}&
 H^0(z)
 :=
 \sup_{a\in A} \big\{az-h(a)\big\},
 \eeaa
and we assume for simplicity that the supremum is attained by the unique optimal response $\widehat a(z)$. Then, similar to the example from the previous section, the lump sum payment $\xi$ promised at $\tau$ takes the form
 \beaa
 U_{\rm A}(\xi)
 = 
 Y^{Y_0,Z}_\tau
 =
 Y_0+\int_0^\tau Z_t dX_t + \int_0^\tau\big(rY_t-H^0(Z_t)-U_{\rm A}(\pi_t)\big)dt,
 \eeaa
and $Y$ represents the continuation utility of the agent. 

\begin{remark} \label{rem:y=0}
Before continuing, we make the crucial observation that the non-negativity condition on $U_{\rm A}$ and $h$ implies that $Y\ge 0$. As the dynamics of the process $Y$ are given by 
 $$ dY_t=\big(rY_t+h\circ\widehat a(Z_t)-U_{\rm A}(\pi_t)\big)dt+Z_t dW_t^{\widehat a(Z)}, \quad \dbP^{\widehat a(Z)}\mbox{-a.s.}$$ 
 under the optimal response of the agent, we see that $0$ is an absorption point for the continuation utility with optimal effort $\widehat a=0$. 
\end{remark}

By the main reduction result of Theorem \ref{thm:main} we have
 \begin{align*}
   V^{\rm PE}
    &:= \sup_{Z\in\cV}\sup_{\tau\in\cT} 
         \dbE^{\dbP^{\widehat a(Z)}}\left[\int_0^\tau e^{-r t}\big(\widehat a(Z_t)-\pi_t\big)dt
                                  -e^{-r \tau}U_{\rm A}^{-1}\big(Y^{R,Z}_\tau\big)\right],        
 \end{align*}
where 
 $$
  dX_t = \widehat a(Z_t)dt + dW_t^{\widehat a(Z)}
   ~~\mbox{and}~~
  dY_t = \big(rY_t+h\circ\widehat a(Z_t)-U_{\rm A}(\pi_t)\big)dt+ Z_t dW_t^{\widehat a(Z)},
    ~~\dbP^{\widehat a(Z)}\mbox{-a.s.}
 $$
thus leading to a mixed stochastic control and optimal stopping problem with reward function upon stopping (or obstacle) $v_0:=-U_{\rm A}^{-1}$. By classical stochastic control theory, the HJB equation corresponding to this problem is
 \begin{align*}
  0 &= \min\left\{v-v_0,~
                  r v-ryv' -\sup_{\pi}\big\{-\pi-U_{\rm A}(\pi)v'\big\}
                           -\sup_{z}\Big\{\widehat a(z)+h\circ\widehat a(z)v'+\frac12 z^2 v''\Big\}
            \right\} \nonumber \\
    &= \min\left\{v-v_0,~
                  r(v-yv') +\inf_{\pi}\big\{\pi+U_{\rm A}(\pi)v'\big\}
                           -\sup_{a}\Big\{a+h(a)v'+\frac12 \gamma(a)^2v''\Big\}
           \right\}, \quad y\geq 0, \nonumber
 \end{align*}
 by using the inverse optimal response function $\gamma:=\widehat a^{-1}$. 
Finally, it follows from Remark \ref{rem:y=0} together with the definition of the principal problem that the boundary condition at the left boundary of the domain is $v(0)=0$. 
We are then reduced to the obstacle problem
 \bea\label{DPE:Sannokov} 
 0
 =
 \min\!\Big\{ v-v_0
                  \,,\,
                  r (v-yv')+I(v')-J(v',v'')
         \Big\}, 
 ~~
 y\!>\!0
 ~~\mbox{and}~~v(0)=0,
 \eea
where, assuming further that $U_{\rm A}$ is $C^1$ with $U_{\rm A}'(0)=\infty$ and $U_{\rm A}'(\infty)=0$, 
  \begin{align} \label{IJ}
  	I(p):=\big((U_{\rm A}')^{-1}+p U_{\rm A}\circ(U_{\rm A}')^{-1}\big)\Big(\frac{-1}{p}\Big) 
  	 \quad \mbox{and} \quad
  	 J(p,q):=\sup_{a\in A} \Big\{a+h(a)p+\frac12 \gamma(a)^2 q\Big\}.
  \end{align}
 We also refer the interested reader to the recent work \cite{PT2020} for more detailed analysis of this model.

\subsection{An American contracting version of Sannikov \cite{Sannikov08}}

In the context of the previous example, let the agent utility function be such that $U(0)=0$. Given an American contract $\Cbf=(\tau^{\scriptscriptstyle\rm P},\pi)$, the agent problem is defined by:
 \beaa
 V^{\rm A}(\tau^{\scriptscriptstyle\rm P},\pi)
 :=
 \sup_{\tau,\alpha} \dbE^{\dbP^{\alpha}}\left[\int_0^{\tau\wedge\tau^{\scriptscriptstyle\rm P}} e^{-rt}\big(U_{\rm A}(\pi_t)-h(\alpha_t)\big) dt\right]
 \eeaa
where
 \beaa
 X_t=X_0+\int_0^t \alpha_sds+dW^\alpha_s, 
   ~~t\geq 0,~~\dbP^\alpha\mbox{-a.s}.
 \eeaa
The principal chooses optimally the contract by solving:
 \beaa
  V^{\rm PA} :=
  \sup_{\substack{\tau^{\scriptscriptstyle\rm P}, \pi \\ V^{\rm A}(\tau^{\scriptscriptstyle\rm P},\pi) \geq U(R)}} 
   \dbE^{\dbP^{\widehat\alpha}}
  \left[\int_0^{\tau^{\rm P}\wedge\widehat\tau} e^{-rt}\big(\widehat\alpha_t-\pi_t\big)dt\right],
 \eeaa
where $(\widehat\tau,\widehat\alpha)$ denotes the optimal response of the agent to the proposed contract $(\tau^{\scriptscriptstyle\rm P},\pi)$. Applying the result of our main theorem, and following similar calculations as in the previous example, we see that 
 \beaa
  V^{\rm PA} = \sup_{Y_0\ge R}V_0(Y_0),
   &\mbox{where}&
  V_0(Y_0) := \sup_{Z,\pi} \dbE \left[\int_0^{{\rm T}_0} e^{-rt}(\widehat a(Z_t)-\pi_t)dt\right],
 \eeaa
where $\widehat a$ is the maximizer of the Hamiltonian, as defined in the previous example, and ${\rm T}_0:=\inf\{t>0: Y_t\le 0\}$, and the controlled state $Y$ is defined by the dynamics:
 $$ dY_t = \big(rY_t + h\circ\widehat a(Z_t) - U_{\rm A}(\pi_t)\big)dt + Z_t dW^{\widehat a(Z)}, \quad \dbP^{\widehat a(Z)}\mbox{-a.s.} $$
By standard stochastic control theory, we see that the dynamic programming equation corresponding to this problem is
 \beaa
 v(0)=0,
 &\mbox{and}&
 r (v-yv')+I(v')-J(v',v'')=0~~\mbox{on}~~(0,\infty),
 \eeaa
where $I$ and $J$ are defined in \eqref{IJ}. Notice that the last equation differs from \eqref{DPE:Sannokov}  by the absence of the obstacle constraint. 
However, it is shown in \cite{PT2020} that the two equations are equivalent, so that by their uniqueness result, the American contracting version of Sannikov coincides with the original Sannikov contracting problem.

\subsection{An explicit example without optimal contract}

This section illustrates the use of our main result in the context of the European contracting problem. 
In order to gain in simplicity and to favour as most explicit results as possible, the following example intentionally violates the technical conditions of the general contracting problem. 
We refrain from giving a fully rigorous proof of the solution provided in the present example, and we shall point out how our main results may be extended to the present context.

Suppose that the contract has no continuous payment component, and that the agent is solving the simple problem with $\tau:=\tau^{\scriptscriptstyle\rm P}$ possibly taking $\tau=\infty$ with positive probability:
 \begin{align*}
   \sup_{\alpha} \dbE^{\dbP^\alpha}\left[\xi-\frac12\int_0^\tau \alpha_t^2 dt\right],
 \end{align*}
 where $\alpha$ is any progressively measurable process which guarantees the existence of a weak solution $\dbP^\alpha$ for the following SDE:
  $$ X_t=X_0+\int_0^t \alpha_sds+W_t^{\alpha},  \quad 0\leq t \leq\tau, \quad \dbP^\alpha\mbox{-a.s.} $$
Clearly, this requires that $\dbE^{\dbP^0}\big[{\rm D}^{\dbP^\alpha |\dbP^0}_T\big]=1$, so that existence follows from the Girsanov theorem. In the present context, we observe that we also have uniqueness of such a weak solution.

%
%

The Hamiltonian is given by
 \begin{align*}
    H(y,z,\gamma)
    =
    \frac12{\rm Tr}[\gamma]+H^0(z), 
    \quad \mbox{where} \quad
     H^0(z)
    :=
    \sup_{a\in\dbR} \Big\{az-\frac12a^2\Big\}
    =
    \frac12 z^2,
 \end{align*}
and the supremum is attained by the optimal response $\widehat a(z)=z$. In particular, the agent optimal response is unique. In the present setting, the lump sum payment $\xi$ takes the form
 \begin{align*}
   \xi = Y^{Y_0,Z}_\tau 
    &= Y_0+\int_0^\tau Z_t dX_t - \int_0^\tau H^0(Z_t)dt  
     = Y_0+\int_0^\tau Z_t dX_t - \int_0^\tau\frac12 Z_t^2dt.
 \end{align*}
This representation may be proved by means of the standard dynamic programming principle satisfied by the agent dynamic value process together with appropriate transvesality conditions satisfied by the stopping time $\tau$ and the admissible controls. 
This is in fact related to the corresponding backward SDE which allows for possibly infinite stopping, see \cite[Secton 6.3]{CZ12} and \cite{LRTY20, Roy04}. 
Given $\xi = Y_\tau^{Y_0,Z}$, the agent's optimal control is $\widehat\alpha = Z$, and $V_0^E(\tau,\xi) = Y_0$.

Then, the main reduction result of Theorem \ref{thm:main} applies and provides
 \begin{align*}
   V^{\rm PE}
     &= \sup_{(\tau,\xi)\in\mathfrak{C}_R^E}
        \dbE^{\dbP^{\alpha^*}}\bigg[ \int_0^\tau e^{-\beta t}dX_t - e^{-\beta\tau}\xi\bigg] 
      = \sup_{Z\in\cV}\sup_{\tau\in\cT} \dbE^{\dbP^{Z}}\bigg[ \int_0^\tau e^{-\beta t}Z_tdt - e^{-\beta\tau} Y^{R,Z}_\tau\bigg],
 \end{align*}
 where 
 \begin{align*}
 	 dX_t
 	=
 	Z_tdt + dW^Z_t
 	\quad \mbox{and} \quad
 	dY_t
 	=
 	\frac12 Z_t^2 dt+Z_t dW^Z_t,
 	\quad \dbP^{Z}\mbox{-a.s.}
 \end{align*}

By classical stochastic control theory, the HJB equation corresponding to this combined optimal control and optimal stopping problem is
 \begin{align} \label{DPEv}
   0  
    &=
     \min\Big\{v-v_0
                  \,,\,
                  \beta v-\sup_{z\in\dbR}\Big(\frac12z^2(v'+v'')+z\Big) 
         \Big\} 
       ~~\mbox{with}~~v_0(y):= -y,  \\
    &=
     \min\Big\{ v + y
                  \,,\,
                  \beta v+\frac12(v'+v'')^{-1}
         \Big\},
       ~~\mbox{with}~~v'+v''<0, \nonumber
 \end{align}
 where the supremum is attained at 
  \begin{align} \label{eq:hatz}
  	\widehat z(y):= -\frac{1}{v'(y)+v''(y)}.
  \end{align}

By introducing the function $u(s):=s v(\ln{s})$, for $s>0$, we compute that $u'(s)=(v+v')(\ln{s})$ and $u''(s)=\frac1s(v'+v'')(\ln{s})$, thus reducing the last ODE to
 \bea\label{ODEu0}
 0
  =
 \min\Big\{ u-u_0
                  \,,\,
                  \beta u+\frac{1}{2u''}
         \Big\},
 ~~\mbox{with}~~
 u''<0, 
 ~~\mbox{and}~~u_0(s):=-s\ln{s}.
 \eea
Notice that the strict concavity of $u$ together with $\beta u+\frac{1}{2u''}\ge 0$ imply that $u>0$, 
and therefore $u$ must be increasing. 
We may explore the region where the solution $u$ possibly coincides with the obstacle $u_0(s)=-s\ln(s)$:
\bea\label{arret}
 \{u=u_0\}
 \subseteq
 \left\{\beta u_0+\frac{1}{2u_0''}
          \ge 0
 \right\}
 =
 \Big\{-\beta s\ln{s}-\frac{s}{2}
          \ge 0
 \Big\}
 =
 \Big\{s\le s^*:=e^{-\frac1{2\beta}}\Big\}.
\eea

This suggests searching for a solution of \eqref{ODEu0} of the form 
\bea\label{u:form}
 u_n(s)
 =
 \mathbf{1}_{\{s\le s_n\}} u_0(s)
 +\mathbf{1}_{\{s>s_n\}} u(s),
 ~~\mbox{for some}~s_n\in (0,s^*],
\eea
 and some $C^2$ function $u_n\ge u_0$ satisfying
 \bea\label{ODEu}
 \beta \,u_n+\frac{1}{2u_n''}=0
 ~~\mbox{on}~(s_n,\infty),
 &\mbox{with}&
 u_n(s_n)=u_0(s_n),
 ~
 u_n'(s_n)=u_0'(s_n).
 \eea
The last ODE is equivalent to $2\beta u_n''+\frac1{u_n}=0$ which, after multiplying by $u_n'$ and direct integration and using the boundary condition in \eqref{ODEu}, provides
 \bea\label{ODEu'2}
 \beta u_n'(s)^2=c_n-\ln{u_n(s)},~~s\ge s_n, 
 &\mbox{where}&
 c_n:=\beta u'_0(s_n)^2+\ln{u_0(s_n)}.
 \eea  
By the smooth fit condition, we have $u'_n(s_n)=u_0'(s_n)\ge u_0'(s^*) = \frac1{2\beta}-1>0$ for $\beta\in(0,\frac12)$.  
We then search for an increasing candidate solution of the ODE 
 $$ \sqrt{\beta} u_n'(s)=\sqrt{c_n-\ln{u_n(s)}},\quad s\ge s_n. $$ 
Direct integration of this equation provides 
 \beaa
 s-s_n
 =
 \sqrt{\beta}
 \int_{s_n}^s\frac{u_n'(t)}{\sqrt{c_n-\ln{u_n(t)}}} dt
 =
 e^{c_n}\sqrt{\beta\pi}
 \int_{c_n-\ln{u_n(s)}}^{c_n-\ln{u_0(s_n)}} \gamma(t) dt
 ,
 &s\ge s_n,&
 \eeaa
 where $\gamma(t):=\frac{e^{-t}}{\sqrt{\pi t}}$ is the density function of the $\Gamma(\frac12,1)$ distribution.
Denoting by $F$ the corresponding cumulative distribution function, and recalling that $c_n-\ln{u_0(s_n)} = \beta u_0'(s_n)^2$, we see that
\bea\label{us0}
 \ln{u_n(s)}
 =
 c_n-F^{-1}\Big(-\frac{s-s_n}{e^{c_n}\sqrt{\beta\pi}}
                                     +F\big(\beta u_0'(s_n)^2\big)
                             \Big),
 &s\in[s_n,s_n'),&
\eea
where $s_n'$ is the maximum value of $s$ such that the last equation has a solution:
 \bea\label{sn'}
 s_n':=s_n + e^{c_n}\sqrt{\beta\pi}F\big(\beta u_0'(s_n)^2\big),
 &\mbox{and}&
  u(s_n')=e^{c_n},~~u'(s_n')
 =0.
 \eea 
At this point, we observe that $s'_n<\infty$, so that the maximal increasing solution of the ODE started from an arbitrary $s_n\in(0,s^*]$ is only defined up to the finite point $s'_n$. 
However, if we choose a sequence $s_n$ converging to zero, then $u_0'(s_n)\longrightarrow\infty$ and $c_n\longrightarrow\infty$, so that $s_n'\longrightarrow\infty$. 
For this reason, in order to construct a solution of the ODE on the positive real line, we now set 
 \beaa
 s_n:=\frac1n,
 &\mbox{and we extend $u_n$ to $\dbR_+$ by}&
 u_n(s):=u_n(s'_n)=e^{c_n},
 ~\mbox{for all}~
 n\ge 1,
 \eeaa
 and we argue that the sequence $(u_n)_n$ is increasing. 
Indeed, $u_n>u_0$ on $(s_n,s_n']$ because $u_0''(s_n)<u_n''(s_n)$ as $s_n<s^*$. Then, $u_{n+1}(s_n)>u_n(s_n)$ and by standard comparison of the solution of the ODE, we see that $u_{n+1}>u_n$ on $(s_{n+1},\infty)$. 

Consequently, there exists a strictly concave increasing function $u$ on $\dbR_+$, such that
 \bea
 u_n\longrightarrow u,
 &\mbox{pointwise and uniformly on compact subsets of}&
 \dbR_+,
 \eea
by the Dini theorem. This limiting function satisfies
 \beaa
 u(0)=0,~u'(0)=\infty,~u(\infty)=\infty,~u'(\infty)=0,
 &\mbox{and}&
 u>u_0,~\beta u+\frac1{2u''}=0~\mbox{on}~(0,\infty),
 \eeaa
and therefore induces the required classical solution $v(y)=e^{-y}u(e^y)$ of the dynamic programming equation \eqref{DPEv}. Figure \ref{fig:test1} shows a numerical result of $u$ and $u_0$ with $\beta=0.05$.
Finally, by following a classical verification argument, 
	we may show that the solution of the optimal control-stopping problem is 
	 $$ \widehat z\big(\widehat Y_s\big),  \quad \widehat{\tau}:=\inf\Big\{t>0: \widehat Y_t=-\infty\Big\} = \infty, $$
	 with the optimal controlled dynamics
	  \begin{equation*}
	  \widehat{Y}_t = R+\int_0^t\widehat z\big(\widehat Y_s\big)dX_s - \int_0^t\frac12\widehat z\big(\widehat Y_s\big)^2ds,
	  \end{equation*}
	where $\widehat z$ is defined in \eqref{eq:hatz}, and the value of the problem is $v(R)$. See Figure \ref{fig:test2} for the numerical result of $\widehat{z}$. 
	We conclude that in this example there is no optimal contract with finite terminal time. 

\begin{figure}
	\centering
	\begin{minipage}{.5\textwidth}
		\centering
		\includegraphics[width=.95\linewidth]{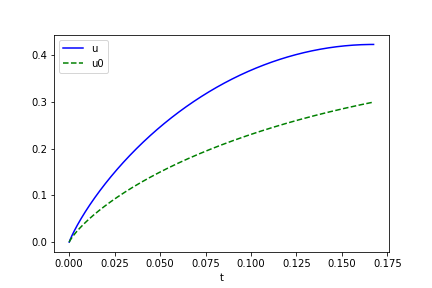}
		\captionof{figure}{functions $u$ and $u_0$}
		\label{fig:test1}
	\end{minipage}%
	\begin{minipage}{.5\textwidth}
		\centering
		\includegraphics[width=.95\linewidth]{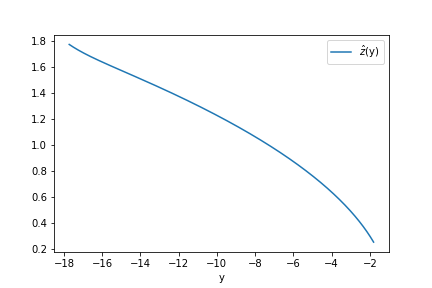}
		\captionof{figure}{optimal control $\widehat{z}(y)$}
		\label{fig:test2}
	\end{minipage}
\end{figure}

\section{Density of revealing contracts} 
\label{sec:proof}

For $(t,\omega)\in\llbracket 0, \tau\rrbracket$, define
  \begin{equation*}  \label{Sigma}
    \Sigma_t(\omega,b):=(\sigma_t\sigma_t^\top)(\omega,b) \quad \mbox{and} \quad {\bf \Sigma}_t(\omega) := \big\{\Sigma_t(\omega,b)\in\cS_d^+(\dbR) : b\in B\big\}.
  \end{equation*}
We also introduce the inverse map which assigns to every squared diffusion $\Sigma\in{\bf \Sigma}_t(\omega)$ the corresponding set of generating controls
  \begin{equation*}
    {\bf B}_t(\omega,\Sigma):=\big\{b\in B: (\sigma_t\sigma_t^\top)(\omega,b)=\Sigma\big\}. 
  \end{equation*}
This allows us to isolate the partial maximization with respect to the squared diffusion in the Hamiltonian $H$ in \eqref{eq:defH}: 
  \begin{align*}
    H_t(\omega,y,z,\gamma) = \sup_{\Sigma\in{\bf\Sigma}_t(\omega)}\left\{F_t(\omega,y,z,\Sigma) + \frac{1}{2}{\rm Tr}[\Sigma\gamma]\right\},
  \end{align*}
  where
  \begin{align*}
    F_t(\omega,y,z,\Sigma):= \sup_{(a,b)\in A\times{\bf B}_t(\omega,\Sigma)}\big\{-c_t(\omega,a,b)-k_t(\omega,a,b)y + \sigma_t(\omega,b)\lambda_t(\omega,a)\cdot z\big\}.
  \end{align*}
We see that $2H$ is the convex conjugate of $-2F$.
Let $\Sigma_t(\omega,b)^{\frac{1}{2}}$ denote the corresponding square root and consider 
  \begin{equation} \label{SDESigma}
   X_t = X_0 + \int_0^t\Sigma_r(\omega,\beta_r)^{\frac{1}{2}}dW_r.
  \end{equation}
Clearly, any weak solution $(\dbP,\beta)$ of \eqref{SDESigma} is also a solution of \eqref{SDEsigma}. 
Let
  \begin{align*}
    \cP^o := \big\{\dbP^o\in\mathfrak{M}_+^1(\Omega): (\dbP^o,\beta) \mbox{ is a weak solution of \eqref{SDEsigma} for some }\beta\big\},
  \end{align*}
  and notice that for any weak solution $(\dbP^o,\beta)$ of \eqref{SDESigma} we have that for $\dbP^o$-almost every $\omega\in\Omega$
  \begin{align*}
    \widehat\sigma^2_t(\omega)\in{\bf \Sigma}_t(\omega) \quad \mbox{and} \quad \beta_t(\omega)\in{\bf B}_t\big(\omega,\widehat{\sigma}_t^2(\omega)\big).
  \end{align*}
For any fixed diffusion coefficient, there is a one-to-one correspondence between the solutions of \eqref{Xalphabeta} and \eqref{SDEsigma} through Girsanov's theorem. 
Define 
  $$ \cU^o(\dbP^o):=\big\{\nu=(\alpha,\beta), \,\dbF\mbox{-optional: } \alpha_t(\omega)\in A(\omega),\,\beta_t(\omega)\in \mathbf{B}_t\big(\omega,\widehat{\sigma}_t^2(\omega)\big)\,\mbox{ on }\dbR_+, \,\dbP^o\mbox{-a.s.}\big\}, $$
  and 
  $$ \cM^o:=\big\{(\dbP^o,\nu):\,\dbP^o\in\cP^o \mbox{ and } \nu\in \cU^o(\dbP^o)\big\}. $$
Notice that we have a one-to-one correspondence between the set of control models $\cM$ and the set $\cM^o$ by means of Girsanov's theorem. 
We may rewrite the agent's problem 
 \begin{equation*}
   V^{\rm E}(\Cbf) = \sup_{\dbP^o\in\cP^o}V^{\rm E}(\Cbf,\dbP^o)
 \end{equation*}
 with 
 \begin{equation*}
   V^{\rm E}(\Cbf,\dbP^o) := \sup_{\nu\in\cU^o(\dbP^o)}\dbE^{\dbP^{\nu}}\left[K_\tau^\nu U_{\rm A}(\xi) + \int_0^\tau\cK_r^\nu\big(U_{\rm A}(\pi_r)-c_r(\nu_r)\big)dr\right].
 \end{equation*}
  where the measure $\dbP^\nu$ is defined by the Girsanov transformation 
   $$ \frac{d\dbP^\nu}{d\dbP^o}\bigg|_{\cF_t} = \cE\left(\int_0^\cdot\lambda_r(\alpha_r)\cdot dW_r\right)_t, \quad t\geq 0. $$
 
\vspace{3mm}
 
We now provide a representation of the agent's value function by means of second-order backward SDEs (2BSDEs) as introduced by Soner, Touzi and Zhang \cite{STZ12}. 
We apply our recent development of 2BSDE with random horizon and without the regularity conditions \cite{LRTY20}, based on the work of Possama\"i, Tan and Zhou \cite{PTZ18}.

Given a final payment $\xi$, we consider the 2BSDE
 \begin{equation} \label{2BSDE}
   Y_{t\wedge\tau} = U_{\rm A}(\xi) + \int_{t\wedge\tau}^\tau \big(F_s(Y_s,Z_s,\widehat{\sigma}^2_s) + U_{\rm A}(\pi_s)\big)ds - \int_{t\wedge\tau}^\tau Z_s\cdot dX_s + \int_{t\wedge\tau}^\tau dK_s, \quad \dbP^o\mbox{-a.s.}, 
 \end{equation}
 for each $\dbP^o\in\cP^o$. 

\begin{definition} \label{def:2BSDE}
  For $1<p<q$ and $-\mu\leq\eta<\rho$, the process $(Y,Z,K)\in\cD^p_{\eta,\tau}\big(\cP^o, \dbF^{+,\cP^o}\big)\times \cH^p_{\eta,\tau}\big(\cP^o,\dbF^{\cP^o}\big)\times \cI^p_{\eta,\tau}(\cP^o,\dbF^{\cP^o})$ is the solution of the 2BSDE \eqref{2BSDE}, 
   if 
   \begin{itemize}
    \item for each $\dbP^o\in\cP^o$, $(Y,Z,K)$ satisfies \eqref{2BSDE} $\dbP^o$-a.s.
    \item the nondecreasing process $K$ satisfies the minimality conidtion: for all $\dbP^o\in \cP^o$  
            \begin{equation*}
               K_{t_1\wedge \tau} = \essinf^{\dbP^o}_{\dbP'\in\cP_+(t_1\wedge\tau,\dbP^o)}\mathbb{E}^{\dbP'}\Big[K_{t_2\wedge\tau}\Big|\mathcal{F}^{+, \dbP'}_{t_1\wedge\tau}\Big], 
            \end{equation*}
            where 
             $$  \cP^o_+(\sigma, \dbP):=\bigcup_{h>0}\cP^o\big((\sigma+h)\wedge \tau, \dbP\big), \quad \cP^o(\sigma, \dbP):=\big\{\dbP'\in\cP^o: \dbP'=\dbP\mbox{ on }\cF_\sigma\big\}. $$ 
   \end{itemize}
\end{definition}
 
The definition of 2BSDE here is slightly different from that in \cite{LRTY20}: the nondecreasing process $K$ is assumed to be aggregated, i.e., $K$ is given as a unique process, and not as a family of processes indexed by $\cP^o$. 
Indeed, in general a family of processes $\big\{K^{\dbP^o}\big\}_{\dbP^o\in\cP^o}$ is given through a nonlinear Doob-Meyer or optional decomposition theorem, applied under each $\dbP^o\in\cP^o$. 
Under the usual set-theoretic Zermelo-Fraenkel set theory (ZFC) framework and the continuum hypotheses, as in Nutz \cite{Nut12}, the stochastic integral $\int_0^tZ_s\cdot dX_s$ can be defined pathwisely on $\Omega$ without the need for exclusion of any null set
  and therefore does not depend on $\dbP^o$. 
Consequently, $K$ does not depend on $\dbP^o$. In other words, $\big\{K^{\dbP^o}\big\}_{\dbP^o\in\cP^o}$ can be aggregated into the resulting medial limit $K$, i.e., $K^{\dbP^o}=K$, $\dbP^o$-a.s.~for all $\dbP^o\in\cP^o$. 

\begin{proposition}
  For all $\Cbf\in\frakC$ the 2BSDE \eqref{2BSDE} has a unique solution. 
\end{proposition}

\begin{proof}
  For $(t,\omega)\in\dbR_+\times\Omega$ with $t\leq\tau(\omega)$ we introduce the dynamic versions $\cM^o(t,\omega)$ and $\cP^o(t,\omega)$ of the sets $\cM^o$ and $\cP^o$ by considering the SDE \eqref{SDESigma} on $\llbracket t,\tau\rrbracket$ starting at time $t$ from the path $\omega\in\Omega$. 
  
  \vspace{2mm}
  
  \noindent \textbf{(i).} We first show that the family $\big\{\cP^o(t,\omega):(t,\omega)\in\llbracket 0,\tau\rrbracket\big\}$ is saturated, i.e., for all $\dbP_1^o\in\cP^o(t,\omega)$ we have $\dbP_2^o\in\cP^o(t,\omega)$ for every probability measure $\dbP_2^o\sim\dbP_1^o$ 
    such that $X$ is a $\dbP_2^o$-local martingale. 
  To verify this, notice that the equivalence between $\dbP_1^o$ and $\dbP_2^o$ implies that the quadratic variation of $X$ is not changed by passing from $\dbP_1^o$ to $\dbP_2^o$. 
  As $X$ is a $\dbP_2^o$-local martingale, it follows that if $(\dbP_1^o,\beta)\in\cM^o(t,\omega)$, then $(\dbP_2^o,\beta)\in\cM^o(t,\omega)$.
  
  \vspace{2mm}
  
  \noindent \textbf{(ii).} We next verify that the generator $F_s(Y_s,Z_s,\widehat{\sigma}^2_s) + U_{\rm A}(\pi_s)$ satisfies the conditions of Lipschitz-continuity, monotonicity, and integrability. 
    For all $(t,\omega)\in\llbracket 0,\tau\rrbracket$ and $\Sigma\in\mathbf{\Sigma}_t(\omega)$,
      \begin{align*}
        |F_t(\omega,y,z,\Sigma)-F_t(\omega,y',z',\Sigma)| &\leq \|k_t\|_\infty|y-y'| + \|\lambda_t\|_{\infty}\sup_{b\in\mathbf{B}_t(\omega,\Sigma)}\big|\sigma_t(\omega,b)^\top(z-z')\big|  \\
            &= \|k_t\|_\infty|y-y'| + \|\lambda_t\|_{\infty}\big|\Sigma^\frac{1}{2}(z-z')\big|,
      \end{align*}  
      for $(y,z),(y',z')\in\dbR\times\dbR^d$, and 
      \begin{align*}
        (y-y')\big(F_t(\omega,y,z,\Sigma)-F_t(\omega,y',z,\Sigma)\big) &\leq (y-y')\sup_{(a,b)\in A\times\mathbf{B}_t(\omega,\Sigma)}\big\{-k_t(\omega,a,b)(y-y')\big\} \\
                          & = -\inf_{(a,b)\in A\times\mathbf{B}_t(\omega,\Sigma)}k_t(\omega,a,b)(y-y')^2, 
      \end{align*}
      for $y,y'\in\dbR$. 
    As $k,\sigma,\lambda$ are bounded, the generator is Lipschitz-continuous in $(y,z)$ and monotone in $y$. 
    Notice that
      $$ F_t(\omega,0,0,\Sigma) = \sup_{(a,b)\in A\times\mathbf{B}_t(\omega,\Sigma)}\{-c_t(\omega,a,b)\}. $$
    For $f^0_s(\omega):=F_s(\omega,0,0,\widehat{\sigma}_s^2)+U(\pi_s)$ and 
      \begin{align*}
         f^{0,t,\omega}_s(\omega') &:= F_{t+s}\big(\omega\otimes_t\omega',0,0,\widehat{\sigma}_s^2(\omega')\big)+U_{\rm A}\big(\pi_{t+s}(\omega\otimes_t\omega')\big)  \\
                         &\hspace{1mm} = \sup_{(a,b)\in A\times\mathbf{B}_{t+s}(\omega\otimes_t\omega',\widehat{\sigma}^2_s(\omega'))}\big\{-c_{t+s}(\omega\otimes_t\omega',a,b)\big\} + U_{\rm A}\big(\pi^{t,\omega}_{s}(\omega')\big)
      \end{align*}
      we obtain for $\vec{\tau}=\tau^{t,\omega}-t$ that
      \begin{align*}
        & \cE^{\cP^o(t,\omega)}\bigg[\bigg(\int_0^{\vec{\tau}} e^{2\rho r}\big|f^{0,t,\omega}_r\big|^2dr\bigg)^{\frac{q}{2}}\bigg] \\
        & \quad \leq C_q \left\{\cE^{\cP^o(t,\omega)}\bigg[\bigg(\int_0^{\vec{\tau}} e^{2\rho r}\big|\overline{c}^{t,\omega}_r\big|^2dr\bigg)^{\frac{q}{2}}\bigg] 
                              + \cE^{\cP^o(t,\omega)}\bigg[\bigg(\int_0^{\vec{\tau}} e^{2\rho r}\big|U_{\rm A} (\pi_r^{t,\omega})\big|^2dr\bigg)^{\frac{q}{2}}\bigg]\right\} <\infty, 
      \end{align*}
      where the last inequality follows from Assumption \ref{assum:costc} and \eqref{contract-growth}.
    
    The dynamic programming requirements of \cite[Assumption 2.1]{PTZ18} and \cite[Lemma 6.6]{LRTY20} follow from the more general results given in El Karoui and Tan \cite{ElKTan13,ElKTan15}.
    
    Finally, as $\xi$ satisfies the integrability condition \eqref{contract-growth}, the required well-posedness result is a direct consequence of  \cite[Theorem 3.3]{LRTY20}.
\end{proof}

Now, we have the relation of the agent's problem and 2BSDE. 

\begin{proposition} \label{prop:optimality.et.K}
  Let $(Y,Z,K)$ be the solution of the 2BSDE \eqref{2BSDE}. 
  Then, we have 
    $$ V^{\rm E}(\Cbf) = \sup_{\dbP^o\in\cP^o}\dbE^{\dbP^o}[Y_0]. $$
  Moreover, $(\widehat\dbP,\widehat\nu)\in\widehat\cM^{\rm E}(\Cbf)$ if and only if 
    \begin{itemize}
     \item $\widehat\nu$ is a maximizer in the definition of $F(Y,Z,\widehat\sigma^2)$, $dt\otimes\widehat\dbP$-a.e.;
     \item $K_\tau=0$, $\widehat\dbP$-a.s.
    \end{itemize}
\end{proposition}

\begin{proof}
 \textbf{(i).} By \cite[Proposition 5.2]{LRTY20}, the solution of the 2BSDE \eqref{2BSDE} can be represented as the supremum of the solutions of BSDEs 
     \begin{equation} \label{eq:2BSDErep}
        Y_0 = \esssup_{\dbP'\in\cP^o_+(0,\dbP^o)}^{\dbP^o}\cY_0^{\dbP'}, \quad \dbP^o\mbox{-a.s.~for all }\dbP^o\in\cP^o, 
     \end{equation}
     where for all $\dbP^o\in\cP^o$, $(\cY^{\dbP^o},\cZ^{\dbP^o})$ is the solution of the following BSDE under $\dbP^o$:
      $$ \cY_0^{\dbP^o} = U_{\rm A}(\xi) + \int_0^\tau\Big( F_r\big(\cY^{\dbP^o}_r,\cZ^{\dbP^o}_r,\widehat\sigma_r^2\big) + U_{\rm A}(\pi_r)\Big) dr - \int_0^\tau\cZ^{\dbP^o}_r\cdot dX_r -\int_0^\tau dM^{\dbP^o}_r, \quad \dbP^o\mbox{-a.s.} $$
     with a c\`adl\`ag $(\dbF^{+,\dbP^o},\dbP^o)$-martingale $M^{\dbP^o}$ orthogonal to $X$. 
    For all $(\dbP^o,\nu)\in\cM^o$, consider the linear BSDE with 
     \begin{equation}  \label{eq:controlBSDE}
       \begin{aligned}
         \cY^{\dbP^o,\nu}_0 = \xi &+ \int_0^\tau\big(-c^\nu_r - k_r^\nu\cY_r^{\dbP^o,\nu} + \sigma^\beta_r\lambda^\alpha_r\cdot\cZ_r^{\dbP^o,\nu} + U_{\rm A}(\pi_r)\big)dr \\ 
                                  &- \int_0^\tau\cZ_r^{\dbP^o,\nu}\cdot dX_r - \int_0^\tau dM^{\dbP^o,\nu}_r, \qquad \dbP^o\mbox{-a.s.},
       \end{aligned}
     \end{equation}
     where $c^\nu_r:=c_r(\nu_r)$ and similar notation apply to $k^\nu$, $\sigma^\beta$, $\lambda^\alpha$. 
    Let $\dbP^\nu$ be the probability measure equivalent to $\dbP^o$ such that 
      $$ \frac{d\dbP^\nu}{d\dbP^o}\bigg|_{\cF_t} = \cE\left(\int_0^\cdot\lambda_r^\alpha dW_r\right)_t, \quad t\geq 0, $$
      where $W$ is a Brownian motion under $\dbP^o$. 
    By It\^o's formula 
     \begin{align*}
      \cK^\nu_\tau\cY^{\dbP^o,\nu}_\tau &= \cY_0^{\dbP^o,\nu} - \int_0^\tau\cK_r^{\nu}k_r^\nu\cY_r^{\dbP^o,\nu} dr + \int_0^\tau\cK_r^\nu\big(c^\nu_r+k^\nu_r\cY_r^{\dbP^o,\nu}-\sigma_r^\beta\lambda^\alpha_r\cdot\cZ_r^{\dbP^o,\nu}-U_{\rm A}(\pi_r)\big)dr \\
          & \hspace{14mm} + \int_0^\tau\cK_r^\nu\cZ_r^{\dbP^o,\nu}\cdot dX_r + \int_0^\tau\cK^\nu_rdM^{\dbP^o,\nu}_r  \\
          &= \cY_0^{\dbP^o,\nu} - \int_0^\tau\cK_r^\nu\big(U_{\rm A}(\pi_r)-c^\nu_r\big)dr + \int_0^\tau\cK_r^\nu\cZ_r^{\dbP^o,\nu}\cdot (dX_r-\sigma_r^\beta\lambda^\alpha_rdr) + \int_0^\tau\cK^\nu_rdM^{\dbP^o,\nu}_r,
     \end{align*}
     where $X_t-\int_0^t\sigma_r^\beta\lambda^\alpha_rdr=\int_0^t\sigma_r^\beta(dW_r-\lambda^\alpha_rdr)$ is a martingale under $\dbP^\nu$. 
    Taking conditional expectation 
     $$ \cY_0^{\dbP^o,\nu} = \dbE^{\dbP^\nu}\left[\cK_\tau^\nu U_{\rm A}(\xi) + \int_0^\tau\cK_r^\nu\big(U_{\rm A}(\pi_r)-c^\nu_r\big)dr\bigg|\cF_0^+\right], \quad \dbP^o\mbox{-a.s.} $$
    Observe that the affine generators in \eqref{eq:controlBSDE} are equi-Lipschitz by boundedness of $k$, $\sigma$, $\alpha$, and there is an $\eps$-maximizer $\widehat\nu^0\in\cU^o(\dbP^o)$ for all $\eps>0$ 
      which obviously induces a weak solution of the corresponding SDE by Girsanov's theorem. 
    This means that the conditions of \cite[Corollary 3.1]{EKPQ97} are satisfied and provides a representation of $\cY_0^{\dbP^o}$ as a stochastic control representation for all $\dbP^o\in\cP^o$ as 
      \begin{equation}  \label{eq:controlBSDE1}
        \cY_0^{\dbP^o} = \esssup^{\dbP^o}_{\nu\in\cU^o(\dbP^o)}\cY_0^{\dbP^o,\nu} 
                       = \esssup^{\dbP^o}_{\nu\in\cU^o(\dbP^o)}\dbE^{\dbP^\nu}\left[\cK_\tau^\nu U_{\rm A}(\xi) + \int_0^\tau\cK_r^\nu\big(U_{\rm A}(\pi_r)-c^\nu_r\big)dr\bigg|\cF_0^+\right], \quad \dbP^o\mbox{-a.s.}
      \end{equation}
    Then, for all $\dbP^o\in\cP^o$, we obtain $\dbP^o$-a.s.
      \begin{align*}
        Y_0 &= \esssup^{\dbP^o}_{(\dbP',\nu)\in\cP^o_+(0,\dbP^o)\times\cU^o(\dbP')}\dbE^{\dbP'^{\nu}}\left[\cK_\tau^\nu U_{\rm A}(\xi) + \int_0^\tau\cK_r^\nu\big(U_{\rm A}(\pi_r)-c^\nu_r\big)dr\bigg|\cF_0^+\right]  \\
            &= \esssup^{\dbP^o}_{(\dbP',\nu)\in\cM^o,\,\dbP'=\dbP^o \textnormal{ on }\cF_0^+}\dbE^{\dbP'^{\nu}}\left[\cK_\tau^\nu U_{\rm A}(\xi) + \int_0^\tau\cK_r^\nu\big(U_{\rm A}(\pi_r)-c^\nu_r\big)dr\bigg|\cF_0^+\right].
      \end{align*}
    By similar arguments as in the proof of \cite[Lemma 3.5]{PTZ18}, we may show that the family 
      $$ \left\{\dbE^{\dbP'^{\nu}}\left[\cK_\tau^\nu U_{\rm A}(\xi) + \int_0^\tau\cK_r^\nu\big(U_{\rm A}(\pi_r)-c^\nu_r\big)dr\bigg|\cF_0^+\right],\, (\dbP',\nu)\in\cM^o\right\} $$
      is upward directed. 
    Therefore, we may conclude that
     \begin{align*}
       \sup_{\dbP^o\in\cP^o}\dbE^{\dbP^o}[Y_0] 
         &= \sup_{\dbP^o\in\cP^o}\dbE^{\dbP^o}\!\!\left[\esssup^{\dbP^o}_{(\dbP',\nu)\in\cM^o,\,\dbP'=\dbP^o \textnormal{ on }\cF_0^+}\dbE^{\dbP'^{\nu}}\!\!\left[\cK_\tau^\nu U_{\rm A}(\xi) + \int_0^\tau\cK_r^\nu\big(U_{\rm A}(\pi_r)-c^\nu_r\big)dr\bigg|\cF_0^+\right]\right] \\
         &= \sup_{\dbP^o\in\cP^o}\sup_{\nu\in \cU^o(\dbP^o)}\dbE^{\dbP^\nu}\left[\cK_\tau^\nu U_{\rm A}(\xi) + \int_0^\tau\cK_r^\nu\big(U_{\rm A}(\pi_r)-c^\nu_r\big)dr\right]  \\
         &= \sup_{\dbP^o\in\cP^o}V^{\rm E}(\Cbf,\dbP^o) = V^{\rm E}(\Cbf).
     \end{align*}
    
    \vspace{3mm}
    
  \noindent \textbf{(ii).} 
%
   Recall that we have one-to-one correspondence between the set of control models $\cM$ and the set $\cM^o$. 
   From \textbf{(i)}, $\big(\widehat\dbP^o,\widehat\nu\big)\in\cM^o$ is optimal if and only if $V^{\rm E}(\Cbf)=\dbE^{\widehat\dbP^o}[Y_0]=\dbE^{\widehat\dbP^{\widehat\nu}}[Y_0]$. 
   Consider $\dbM^o=\big(\widehat\dbP^o,\widehat\nu\big)$, 
    $$ J^{\rm E}\big(\widehat\dbM^o,\Cbf\big) = \dbE^{\widehat\dbP^{\widehat\nu}}\left[\cK^{\widehat\nu}_\tau U_{\rm A}(\xi) + \int_0^\tau\cK^{\widehat\nu}_r\big(U_{\rm A}(\pi_r)-c_r^{\widehat\nu}\big)dr\right]. $$
   Using It\^o's formula and \eqref{2BSDE}, we obtain that 
    \begin{align*}
      J^{\rm E}\big(\widehat\dbM^o,\Cbf\big) 
       &= \dbE^{\widehat\dbP^{\widehat\nu}}[Y_0] + \dbE^{\widehat\dbP^{\widehat\nu}}\left[\int_0^\tau\cK^{\widehat\nu}_r\big(-c_r^{\widehat\nu}-k_r^{\widehat\nu}Y_r + \sigma_r^{\widehat\beta}\lambda_r^{\widehat\alpha}\cdot Z_r - F_r(Y_r,Z_r,\widehat\sigma_r^2)\big)dr\right] \\
       & \hspace{17mm} - \dbE^{\widehat\dbP^{\widehat\nu}}\left[\int_0^\tau\cK^{\widehat\nu}_r dK_r\right].
    \end{align*}
   Therefore, $\big(\widehat\dbP^o,\widehat\nu\big)$ is optimal if and only if $\widehat\nu$ is a maximizer in the definition of $F$, $dt\otimes\widehat\dbP^o$-a.s., and $K_\tau = 0$, $\widehat\dbP^o$-a.s.
\end{proof}

\begin{proof}[Proof of Proposition \ref{prop:density}]
 Let $\Cbf=(\tau,\pi,\xi)\in\frakC_R^{\rm E}$. By definition, $\Cbf\in\frakC$, $\widehat\cM^{\rm E}(\Cbf)\neq\emptyset$ and $V^{\rm E}(\Cbf)\geq R$.
 Consider the 2BSDE \eqref{2BSDE} with $\xi$. 
 Notice that the integrability conditions, Assumption \ref{assum:costc}, and Definition \ref{assum:xi.pi} imply that the 2BSDE admits a unique solution $(Y,Z,K)$ satisfying 
  $$ \|Y\|_{\cD^p_{\eta,\tau}(\cP^o)} + \|Z\|_{\cH^p_{\eta,\tau}(\cP^o)} <\infty, \quad \mbox{for $\eta\in[-\mu,\rho)$ and $p\in(1,q)$.} $$ 
 By Proposition \ref{prop:optimality.et.K}, we have $K_\tau = 0$, $\widehat\dbP$-a.s., for every $\big(\widehat\dbP,\widehat\nu\big)\in\widehat\cM^{\rm E}(\Cbf)$. 
 
 We fix some $\varepsilon>0$ and define the absolutely continuous approximation of $K$ by 
   $$ K^\varepsilon_t:=\frac{1}{\varepsilon}\int^{t}_{(t-\varepsilon)^+}K_sds, \quad t\geq 0. $$
 Clearly, $K^\varepsilon$ is $\dbF^{\cP^o}$-predictable, nondecreasing $\cP^o$-q.s.~and 
   \begin{equation} \label{eq:K^epsilon=0}
     K^\varepsilon_\tau = 0,\quad \widehat\dbP\mbox{-a.s.~for all } \big(\widehat\dbP,\widehat\nu\big)\in\widehat\cM^{\rm E}(\Cbf). 
   \end{equation}
 We now define the process 
   \begin{equation}  \label{eq:Yeps}
     Y^\varepsilon_t := Y_0 - \int_0^t\Big(F_r\big(Y^\varepsilon_r,Z_r,\widehat\sigma^2_r\big) + U_{\rm A}(\pi_r)\Big)dr + \int_0^tZ_r\cdot dX_r - \int_0^tdK^\varepsilon_r. 
   \end{equation} 
 We may verify that $(Y^\varepsilon,Z,K^\varepsilon)$ solves the 2BSDE \eqref{2BSDE} with terminal condition $\xi^\varepsilon := Y_\tau^\varepsilon$ and generator $F(y,z,\widehat\sigma^2)+U_{\rm A}(\pi)$.
 Indeed, as in the proof of the stability of SDEs, since $K^\varepsilon_\tau\leq K_\tau$ and the norms of $K$ and $Y$ are bounded, we may prove that $\xi^\varepsilon$ satisfies the integrability condition. 
 It follows by \eqref{eq:K^epsilon=0} that $K^\varepsilon$ satisfies the required minimality condition. 
 By a priori estimation we obtain that 
   \begin{equation}  \label{eq:est.of.Yeps.Z}
      \|Y^\eps\|_{\cD^{p'}_{\eta',\tau}(\cP^o)} + \|Z\|_{\cH^{p'}_{\eta',\tau}(\cP^o)} <\infty, \quad \mbox{ for $p'\in(1,p)$ and $\eta'\in[-\mu,\eta)$.}
   \end{equation}
 We observe that a probability measure $\dbP$ satisfies $K_{\tau}=0$ $\dbP$-a.s.~if and only if it satisfies $K^\varepsilon_\tau =0$ $\dbP$-a.s.
 
 Define $\Cbf^\e:=(\tau,\pi,\xi^\eps)$. 
 As $K^\eps=K=0$, $\widehat\dbP$-a.s., for any $(\widehat\dbP,\widehat\nu)\in\widehat\cM^{\rm E}(\Cbf)$, we have $Y^\eps = Y$, $\widehat\dbP$-a.s., in particular $\xi^\eps=\xi$, $\widehat\dbP$-a.s.
 Since by Proposition \ref{prop:optimality.et.K} $\widehat\nu$ is a maximizer in the definition of $F(Y,Z,\widehat\sigma^2)$, $\widehat\nu$ is a maximizer in the definition of $F(Y^\eps,Z,\widehat\sigma^2)$, $dt\otimes\widehat\dbP$-a.e., 
    which again implies by Proposition \ref{prop:optimality.et.K} that $(\widehat\dbP,\widehat\nu)\in\widehat\cM^{\rm E}(\Cbf^\eps)$. The reverse direction holds also true. 
 This implies that $\widehat\cM^{\rm E}(\Cbf^\eps)=\widehat\cM^{\rm E}(\Cbf)\neq\emptyset$.
 
 For $(t,\omega,y,z)\in\llbracket 0,\tau\rrbracket\times\dbR\times\dbR^d$, notice that the map 
   $$ \gamma\mapsto H_t(\omega,y,z,\gamma)-F_t\big(\omega,y,z,\widehat\sigma_t^2(\omega)\big) - \frac{1}{2}{\rm Tr}\big[\widehat\sigma_t^2(\omega)\gamma\big] $$
   is surjective on $(0,\infty)$.
 Indeed, it is nonnegative by the definition of $H$ and $F$, convex, continuous on the interior of its domain, and coercive by the boundedness of $\lambda$, $\sigma$, $k$. 
 Let $\dot{K}^\eps$ denote the density of the absolutely continuous process $K^\eps$ with respect to the Lebesgue measure. 
 The continuity allows us to use the classical measurable selection to find an $\dbF$-predictable process $\Gamma^\eps$ such that 
   $$ \dot{K}^\eps_t = H_t(Y^\eps_t,Z_t,\Gamma^\eps_t) - F_t\big(Y^\eps_t,Z_t,\widehat\sigma_t^2\big) - \frac{1}{2}{\rm Tr}\big[\widehat\sigma_t^2\Gamma^\eps_t\big]. $$
 For $\dot{K}_t^{\eps}>0$, this is a consequence of surjectivity. 
 In the case that $\dot{K}_t^{\eps}=0$, as $\widehat\cM^{\rm E}(\Cbf)=\widehat\cM^{\rm E}(\Cbf^\eps)\neq\emptyset$, it follows from Proposition \ref{prop:optimality.et.K} that $\Gamma^\eps_t$ can be chosen arbitrarily, for instance choose $\Gamma^\eps_t=0$.
 It follows by substituting in \eqref{eq:Yeps} that we have the representation for $Y^\eps$
   $$ Y^\varepsilon_t := Y_0 - \int_0^t\big(H_r(Y^\eps_r,Z_r,\Gamma^\eps_r) + U_{\rm A}(\pi_r)\big)dr + \int_0^tZ_r\cdot dX_r + \frac{1}{2}\int_0^t{\rm Tr}\big[\Gamma^\eps_rd\langle X\rangle_r\big], $$
   so that the contract $\Cbf^\eps=(\tau,\pi,\xi^\eps)=\Big(\tau,\pi,U_{\rm A}^{-1}\big(Y_\tau^{Y_0^\eps,Z,\Gamma^\eps}\big)\Big)$ takes the required form \eqref{def:Y}.
 It follows from \eqref{eq:est.of.Yeps.Z} that the controlled process $(Z,\Gamma^\eps)$ satisfies the integrability condition required in the Definition \ref{def:V} (i).
 By the same argument as in \textbf{Step 1.b.}~in the proof of Theorem \ref{thm:main} (i), we obtain for $(\widehat\dbP,\widehat\nu)\in\widehat\cM^{\rm E}(\Cbf^\eps)$ that 
    $$ H_t(Y^\eps_t,Z_t,\Gamma^\eps_t) = h_t(Y^\eps_t,Z_t,\Gamma^\eps_t,\widehat\nu), \quad dt\otimes\widehat\dbP\mbox{-a.e.} $$
 Consequently, the requirement of Definition \ref{def:V} (ii) is satisfied, and therefore $(Z,\Gamma^\eps)\in\cV$ and $\Cbf^\eps\in\frakC$.  
\end{proof}

\bibliography{PAProblem} 
 
\bibliographystyle{abbrv}  

\end{document}